\documentclass[12pt]{amsart}
\usepackage{amsmath,amssymb,amsfonts,latexsym}
\usepackage{mathrsfs}
\usepackage[dvips]{geometry}
\usepackage{array}
\usepackage{epsf}
\usepackage{epsfig}
\newtheorem*{lemma}{Lemma}
\newtheorem*{prop}{Proposition}
\newtheorem*{thm}{Theorem}
\newtheorem*{cor}{Corollary}

\newcommand{\twoheaddownarrow}{\overset{\sim}{\twoheaddownarrow}}

\newcommand{\gr}{\operatorname{gr}}
\newcommand{\nc}{\newcommand}
\nc{\Ker}{\operatorname{Ker}} \nc{\rker}{\operatorname{rKer}}
\nc{\im}{\operatorname{Im}}
\nc{\stab}{\operatorname {Stab}}
\nc{\ann}{\operatorname {Ann}}
\nc{\Id}{\operatorname {Id}}
\nc{\Prim}{\operatorname {Prim}}
\nc{\Real}{\operatorname {Re}}
\nc{\Ext}{\operatorname {Ext}}
\nc{\rad}{\operatorname {rad}}
\nc{\rk}{\operatorname {rank}}
\nc{\Aut}{\operatorname {Aut}}
\nc{\supp}{\operatorname {supp}}

\usepackage{array}
\usepackage{epsf}
\usepackage{epsfig}
\usepackage{xcolor}

\usepackage{tikz,tikz-cd}

\usetikzlibrary{matrix}

\usetikzlibrary{fit}

\usetikzlibrary{matrix,backgrounds}
\pgfdeclarelayer{myback}
\pgfsetlayers{myback,background,main}

\tikzset{mycolor/.style = {line width=1bp,color=#1}}%
\tikzset{myfillcolor/.style = {draw,fill=#1}}%
\usetikzlibrary{arrows,matrix,positioning}
\usepackage{pstricks}

\usetikzlibrary{graphs,graphs.standard}
\usepackage{changepage}
\usepackage{tikz-cd}
\usepackage{mathrsfs}

\makeatletter
\newcommand*{\encircled}[1]{\relax\ifmmode\mathpalette\@encircled@math{#1}\else\@encircled{#1}\fi}
\newcommand*{\@encircled@math}[2]{\@encircled{$\m@th#1#2$}}
\newcommand*{\@encircled}[1]{%
  \tikz[baseline,anchor=base]{\node[draw,circle,outer sep=0pt,inner sep=.2ex] {#1};}}
\usepackage{float}
\begin{document}

\title [Weierstrass Sections and Components]{Parabolic adjoint action, Weierstrass Sections and components of the nilfibre in type $A$}
\author [Yasmine Fittouhi and Anthony Joseph]{Yasmine Fittouhi and Anthony Joseph\\
Donald Frey, Professorial Chair\\
Department of Mathematics\\
The Weizmann Institute of Science\\
Rehovot, 7610001, Israel}
\date{\today}
\maketitle



 \

\textbf{Abstract}.

\

This work is a continuation of [Y. Fittouhi and A. Joseph, Weierstrass Sections for Parabolic adjoint action in type $A$].

Let $G$ be an irreducible simple algebraic group and $B$ a Borel subgroup of $G$. Let $\mathfrak n$ be the Lie algebra of the nilradical of $B$.  Consider an irreducible subgroup $P$ of $G$ containing $B$. Let $P'$ be the derived group of $P$. Let $\mathfrak m$ be the Lie algebra of the nilradical of $P$.

   A theorem of Richardson asserts that the algebra $\mathbb C[\mathfrak m]^{P'}$  of $P$ semi-invariants is multiplicity-free.  It is hence a polynomial algebra on generators which can be taken to be those weight vectors which are irreducible as polynomials.

A linear subvariety $e+V$ such that the restriction map induces an isomorphism of $\mathbb C[\mathfrak m]^{P'}$  onto $\mathbb C[e+V]$  is called a Weierstrass section for the action of $P'$ on $\mathfrak m$.

Here in type $A$ such a section is constructed, but in better form than that given in Sect. 4, loc cit.
  Yet the main difference is a complete change of emphasis from the construction of a Weierstrass section, to its application.

Let $\mathscr N$ be the nilfibre relative to this action. From the construction of a Weierstrass section $e+V$, it is shown that $e \in \mathscr N$.  Then $P.e$ is contained in a unique irreducible component $\mathscr C$ of $\mathscr N$.

  The structure of $e+V$ is used to give a rather explicit description of $\mathscr C$ as a $B$ saturation set, that is of the form $\overline{B.\mathfrak u}$, where $\mathfrak u$ is a subalgebra of $\mathfrak n$ . This algebra is not necessarily complemented by a subalgebra in $\mathfrak n$ and so $\overline{B.\mathfrak u}$ is not necessarily an orbital variety closure (hence Lagrangian) but it can be.

It is shown that $\mathscr C$ need not contain a dense $P$ orbit and this by a purely theoretical analysis. This occurs for an appropriate parabolic in $A_{10}$ and is possibly the simplest example.  In this particular case $\mathscr C$ is not an orbital variety closure.

%
%
%
%
%

\section{Introduction}\label{1}


The ground field is assumed to be the field of complex numbers $\mathbb C$. For every positive integer $n$, we set $[1,n]=\{1,2,,\ldots\,n\}$.

\subsection{Basic Notation}\label {1.1}

Let $G$ be a simple connected algebraic group, choose a Cartan subgroup $H$ and a Borel subgroup $B\supset H$.  Here and elsewhere we use a lower case Gothic letter to denote the Lie algebra of a given algebraic group, for example $\mathfrak g, \mathfrak h, \mathfrak b$.

Let  $\Delta$ (resp. $\Delta^+ \subset \Delta$) be the set of non-zero (resp. positive)  roots  relative to the pair $(G,H)$ (resp. and to $B$) and $\pi \subset \Delta^+$ the corresponding set of simple roots.  Given $\alpha \in \pi$, $s_\alpha$ denote the corresponding simple reflection and $W$ the group generated by the set of simple reflections.

For each $\pi' \subset \pi$, let $P_{\pi'} \supset B$ denote the corresponding parabolic subgroup of $G$.  Let $L_{\pi'}$ (resp. $M_{\pi'}$)  denote the Levi factor  (resp. nilradical) of $P_{\pi'}$.


In the above, subscripts may be sometimes dropped.

Let $P'$ be the derived group of $P$. A consequence \cite [2.2.2]{FJ} of a theorem of Richardson  is that the invariant algebra $\mathbb C[\mathfrak m]^{P'}$ is polynomial on generators which can be specified as being simultaneously irreducible polynomials and weight vectors, with respect to $\mathfrak h$.

The zero set of such a polynomial is a closed irreducible $P$ stable subvariety of $\mathfrak m$ and hence \cite [2.3.4]{FJ} is the closure of an orbital variety, which we call a hypersurface orbital variety.

\subsection{Diagrams}\label {1.2}

In some sense the observation in \ref {1.1} relating invariants to hypersurface orbital varieties describes the invariant generators.  However orbital varieties are notoriously difficult to determine. Thus for example in type $E_8$ we barely know their number and not at all how many are hypersurfaces in the nilradical of a given parabolic.

Assume for the moment that $\mathfrak g$ is simple of type $A$.

In this case the Levi factor $L$ of a given parabolic $P$ (in standard form) is given by a set of blocks down the diagonal.

We may represent these blocks by a set of successive columns $C_i:i=1,2,\ldots,k$ forming a diagram $\mathscr D_\mathfrak m$ (or simply, $\mathscr D$).

Let $c_i$ denote the height of $C_i$.  We call the sequence $(c_1,c_2,\ldots,c_k)$, the composition that determines $P$. It will also mean the set of columns heights in $\mathscr D$.


\

\textbf {Neighbouring Columns}.  Two columns in $\mathscr D$ of height $s$ with no columns of height $s$ in between are called neighbouring columns (of height $s$).

\

\textbf{Fact}. The number of invariant generators with respect to a given parabolic is just the total number of pairs of neighbouring columns.

\

This result is due to Melnikov \cite {M1}, \cite [Thm. 2.9]{JM}. Its proof uses the Robinson-Schensted algorithm. It is given a new and quite different proof in \cite [Cor. 5.2.6] {FJ}.  It means for example that the number of invariant generators for the composition $(2,1,1,2)$ and for the corresponding partition $(2,2,1,1)$ are the same.  This also results from \cite [Prop. 3.4.8]{FJ}, which is stated in a form valid for all types. However the generators themselves as well as the Weierstrass sections change \textit{dramatically}.

In type $A$, the invariant with respect to a given pair of neighbouring columns is the appropriate Benlolo-Sanderson invariant \cite {BS}. In \cite [Sect. 3]{FJ} we described this invariant more conveniently and more compactly as the restriction of an appropriate minor obtained by taking some derivative of a particular $\mathfrak {sl}(m)$ invariant, for some choice of $m$.


In addition to describing the invariants and their number in type $A$ one finds in \cite {JM}, a simple highest weight module whose zero variety is the given hypersurface orbital variety.  This construction uses using notably the well-known Jantzen sum formula.

The above results were generalized in what can only be described as a tour de force, by Elena Perelman \cite {Pe} to all classical types. In this, McGovern's work \cite {Mc} on describing orbital varieties was used, as well as a presentation of the invariants by the natural embedding in type $A$.

Our presentation of invariants in type $A$ briefly noted above has the advantage that it generalizes to all types \cite [Sect. 3]{FJ}. Yet outside type $A$,  we do not know if this gives all the invariants. At least for $\mathfrak g$ classical this seems extremely likely and can in principle be checked from the work of Perelman \cite {Pe}.

 Ultimately we would like to use this as a basis for a unified treatment for all types.
 \subsection{Weierstrass Sections}\label {1.3}

  A Weierstrass section $e+V$ for the action of $P'$ on $\mathfrak m$ is an affine translate of a vector subspace $V \subset \mathfrak m$ such that the algebra $\mathbb C[\mathfrak m]^{P'}$ of $P'$ invariants on $\mathfrak m$ maps isomorphically on restriction  to the algebra of regular functions on $e+V$.

   A Weierstrass section provides a linearization of the invariant generators.  It also provides a ``canonical form'', or representative, of each orbit passing through $e+V$.  In the special case of $\mathfrak {sl}(3)$ acting on its ten dimensional simple module, one recovers Weierstrass canonical form for elliptic curves, as noted by Popov and Vinberg \cite {PV}, though there was a subtle point that they missed \cite [Sect. 2]{J3} involving the needed irreducibility of the nilfibre (see also \ref {1.5}).  Generally for such problems the nilfibre is hardly ever irreducible.  It is not know if it is necessarily equidimensional.

 \subsection{The Kostant Section}\label {1.4}

 In a seminal paper \cite {K}, partly inspired by the work of Coleman \cite {C}, Kostant provided a construction of a Weierstrass section for the coadjoint action of $\mathfrak g$.   Here $e$ is taken to generate a co-adjoint nilpotent orbit of maximal dimension and through the Jacobson-Morosov theorem one obtains $h \in \mathfrak g$, for which $e$ is an eigenvector of eigenvalue $-1$. Then $V$ is taken to be an $h$ stable complement to $\mathfrak g.e$.  This ensures the injectivity of the restriction map; but surjectivity required Kostant to establish a matching of eigenvalues of $h$ on $V$, with the degrees of generators.

  \subsection{Adapted Pairs}\label {1.5}

 For the action of a connected algebraic group $Q$ (not necessarily semisimple) by morphisms on a vector space $M$,  the nilfibre $\mathscr N$ is defined to be the zero set of the augmentation of the invariant ring $\mathbb C[M]^Q$.  It is the set of so-called nilpotent elements if $Q$ is semisimple and $M=\mathfrak q^*$.

 If $x \in M$ generates an $Q$ orbit of maximal dimension $d$, we say that $x$ is regular. However this concept of regularity is not so useful here.

 Suppose $(e,h)$ with $e \in V$ and $h \in \mathfrak q$ semisimple, with $h.e = -e$.  Then it is easily shown (\cite [1.3]{FJ} that $e \in \mathscr N$.

 It is immediate from its definition that $\mathscr N$ is $Q$ stable.

 We say that $e\in \mathscr N$ is regular in $\mathscr N$  if $Q.e$ is dense in an irreducible component $\mathscr C$ of $\mathscr N$.

 In this case we shall say that $(e,h)$ is an adapted pair.  Adapted pairs can be notoriously difficult to find.

 If $\mathscr N$ is equidimensional of dimension $d$, then $Q.e \subset \mathscr N$ has dimension $d$ if and only if $e \in \mathscr N$ is regular.

 However a parabolic $P$ acts densely on the Lie algebra $\mathfrak m$ of its niradical, whilst a component of $\mathscr N$ will in general have dimension strictly less than $\dim \mathfrak m$.

 By contrast, for coadjoint action it is often the case that the maximal orbit dimension in $\mathfrak q^*$ and in $\mathscr N \subset \mathfrak q^*$ are the same.

 \subsection{Adjoint Action}\label {1.6}

For the adjoint action of a parabolic $\mathfrak p$ on its nilradical $\mathfrak m$, the existence of an adapted pair only occurs for some special choices.

Nevertheless in \cite {FJ} we were able to construct a Weierstrass section for all pairs $(\mathfrak p, \mathfrak m)$ in type $A$.  The method is different to both the Kostant procedure and that used in \cite {J1}.

We consider it to be quite remarkable that our method works \textit{uniformly} in all cases for type $A$.

The general idea of our construction is the following.

For a given invariant, which we recall is the restriction of a minor, we choose $e+V$ so that all of its monomials except one of them vanish on restriction to $e+V$ and the remaining monomial restricts to an element of $V$.  This is easy enough for just one invariant, but our construction must achieve this result for all invariants \textit {simultaneously}.

To obtain the required result we number the boxes in $\mathscr D_\mathfrak m$ with entries in $[1,n]$ so that entries increase down the columns and then on going from left to right.  This forms a tableau $\mathscr T_\mathfrak m$ (which in general is non-standard).  This construction has the property that a line from left to right joining boxes with entries $i,j$ defines a co-ordinate vector $x_{i,j}\in \mathfrak m$.  This line will be denoted by $\ell_{i,j}$, or simply, $(i,j)$.

These lines are concatenated to form a ``composite line'', going monotonically from left to right


Between any two neighbouring columns of height $s$, we require ($(P_1)$ of \ref {5.2}) that all the boxes in the rectangle between these two columns and on the first $s$ rows lie on a unique disjoint union of composite lines, that is to say passing through different boxes.

This union of composite lines defines the required monomial of the invariant, specified by the pair of neighbouring columns, as the product of all the co-ordinates defined by each of the individual lines.  Then all but one of the co-ordinate vectors are evaluated at $1$ (Condition $(P_2)$ of \ref {5.2}).  The evaluation of the invariant is a linear function, deemed to be a generator of $V$ (\cite [Lemma 4.2.5]{FJ}).

To obtain this for all invariants \textit{simultaneously} we use a system of ``gating'' which redirects certain paths - see \cite [4.2.4, Example] {FJ}, in a similar fashion to ``railway points'' rerouting trains.  This can have profound consequences\footnote{Kaiser Wilhelm II declared that preparations against France must be stopped; but his chief of staff answered that this was impossible. It would involve the rerouting of 11,000 trains \cite [Germany declares war, p. 100]{T}.}. In our case we redirect the straight lines used by Ringel et al \cite {BHRR} which had been used to give an explicit dense $P$ orbit in $\mathfrak m$, to now give the much less obvious Weierstrass section.

 In \cite [Section 4] {FJ} we used a system of ``minimal'' gating.  Here we change the construction slightly to ``maximal'' gating.  This is more natural but its main intended purpose was to ensure that $e$ is regular, that is to say generates a dense $P$ orbit in the component in which it is contained.  The way that gating works needs a little thought to understand so patience is necessary!  It resulted from an imaginative insight provided by the first author.

This construction is carried out in detail in Section \ref {5}, without relying on the description in \cite [Sect. 4]{FJ}.  Furthermore we analyze in some detail what this construction gives.  In particular that it behaves well (\ref {5.4.9}) for the adjoining of columns on the right (which is false for the minimal gating). It works less well (\ref {5.4.10}) for adjoining of columns on the left.

In both constructions $e$ (resp. $V$) is always a sum of root vectors (resp. root subspaces).

 \subsection{Change of Emphasis}\label {1.7}

 There is a radical change of emphasis in the present paper compared to \cite {FJ}.  Here we are not so much concerned with Weierstrass sections per se, but rather to where they lead.

 Here we recall that every orbital variety closure can be given the form $\overline {B.(\mathfrak n\cap w(\mathfrak n))}$, where $\mathfrak u :=\mathfrak n\cap w(\mathfrak n)$ is a subalgebra of $\mathfrak n$  complemented in $\mathfrak n$ by a subalgebra.  This implies \cite [Lemma 7.5]{J0} that the intersection of $\overline {B.\mathfrak u}$ with the unique dense $G$ orbit $\mathscr O$ in $G.\mathfrak u$, is involutive and hence by the dimension equality noted in \cite [2.3.1, Eq.(1)]{FJ}, a Lagrangian subvariety of $\mathscr O$.

  One of the remarkable consequences of our construction in \ref {5.4} is that we are able to show (Cor. \ref {6.9.8}) that $e$ above belongs to a ``$B$ saturation set" and that this set is  the component $\mathscr C$ of $\mathscr N$ containing  $e$.

   More precisely, let $\mathfrak n$ be the nilradical of $\mathfrak b$.  Then we describe rather explicitly a subalgebra $\mathfrak u$ of $\mathfrak n$ such that $\overline {B.\mathfrak u}$ is a component $\mathscr C$ of $\mathscr N$ containing $e$.  Generally the subalgebra $\mathfrak u$ we construct is not complemented in $\mathfrak n$ by a subalgebra. Yet all we need for $\overline {B.\mathfrak u}$ to be an orbital variety closure is the dimension estimate
   $\dim \overline {B.\mathfrak u} \geq \frac{1}{2}\dim G.e$ which often holds but can also fail (see the examples in \ref {6.10.9}). Of course $B.\mathfrak u$ does not determine $\mathfrak u$, nor is the choice of $w$ in the previous presentation unique.

   An important aspect of the above analysis is the use of properties of $e$ established in our construction of a Weierstrass section (\ref {5.4}).  Here the key is Lemma \ref {6.9.6} and its use in Theorem \ref {6.9.7}.

 \subsection{Regularity}\label {1.8}

The regularity of $e$ in $\mathscr N$ generally fails but we shall postpone a full discussion of this delicate point for another paper and limit ourselves to the following comments.

A difficulty in computing $\mathfrak p.e$, is that we are \textit {not} allowed to use the same element of $\mathfrak p$ more than once.  This issue already arose in the work of Victoria Sevostyanova \cite {Se}.

A pair $x_{i,j},x_{k,\ell}$ of co-ordinate vectors occurring in $e$ will be called a VS pair if $x_{j,k}$, lies in $\mathfrak n$ or is a root vector of the Levi factor. Notice that the first condition implies that $i<j,k<\ell$.

Given a VS pair, commutation of $x_{i,j}$ with $x_{j,k}$ gives $x_{i,k}$ whilst commutation with of $x_{k,\ell}$ with $x_{j,k}$ gives $x_{j,\ell}$.

In a well-defined sense it can happen that either $x_{i,k}$ or $x_{j,\ell}$ already occurs in $\mathfrak p.e$ or lies in $V$ and we do not need to use $x_{j,k}\in \mathfrak p$ twice.  When this is \textit{not} the case we say that the VS pair is bad.  For example for the parabolic defined by the composition $(3,2,1,1,2,1)$, one finds that $x_{4,6},x_{9,10}$ is a bad VS pair.

On closer inspection, the notion of a bad VS pair is seen to be  a delicate issue - see \ref {6.10.3}, Example.  The classification of all such pairs is even more difficult. These issues will be taken up in a subsequent paper.  For the moment we just give without proof an elegant and difficult result in this direction (Theorem \ref {6.10.3}). It will serve as a guide.

Define $e_{VS}$ by adjoining the right hand co-ordinate vector $x_{j,\ell}$ to $e$ for every bad VS pair. Miraculously this choice ensures that $e_{VS} \in \overline {B.\mathfrak u}$, whilst this can fail for the left hand co-ordinate vector $x_{i,k}$.  This will be proved in a subsequent paper.  Here we just check the special cases we need.

Let $E_{VS}$ denote the sum of the root subspaces defined by the root vectors occurring in $e_{VS}$.  It can happen that these roots are not all linearly independent (\ref {6.10.7}).

We say that a sum $e$ of root vectors in a component $\mathscr C$ of the nilfibre is weakly regular if the corresponding sum of root subspaces $E$ lies in $\mathscr C$ and that $P.E$ is dense in $\mathscr C$.



  In a subsequent paper we shall show (see Remark \ref {6.10.4}) that $P.E_{VS}$ is dense in the component $\mathscr C$ of $\mathscr N$ containing $e_{VS}$. In other words $e_{VS}$ is weakly regular in $\mathscr C$.

 Yet the said linear dependence means that $H.e_{VS}$ is not dense in $E_{VS}$, so $P.e_{VS}$ may fail to be dense in $\mathscr C$ .

By this means we are able to show quite remarkably that (Lemma \ref {6.10.7}) that there are components of the nilfibre with \textit {no} dense $P$ orbit.  This is disappointing but not too surprising since Melnikov \cite {M2} had earlier shown that there are orbital varieties with no dense $P$ orbit (for the obvious choice of $P$) thus negating a conjecture of P. S. Smith.  Her example was in $\mathfrak {sl}(9)$ and required a computer. Our present  example of a component of the nilfibre not admitting a dense $P$ orbit, is in $\mathfrak {sl}(11)$. The computation is relatively easy and does not need  a computer. Yet this component is not an orbital variety closure (see Example 3 of \ref {6.10.9}).  We found a further example (but not given here) in $\mathfrak {sl}(19)$, also not an orbital variety closure.

  \subsection{Components}\label {1.9}

  In \ref {6.10.8} we conjecture (read suggest) how $\mathscr N$ decomposes into irreducible components.  In a subsequent paper we hope to show these components are again $B$ saturation sets, which generalize those obtained from our construction noted above.

  In the case that $e_{VS}$ in $\mathscr N$ is regular, then it can be automatically embedded in an adapted pair $(e,h)$, though unfortunately with $h \in \mathfrak h$ rather than in $\mathfrak h':=\mathfrak h \cap \mathfrak p'$.
  For the moment it is still quite unclear how we could reason backwards to recover a Weierstrass section and this is even less clear if $e_{VS}$ is only weakly regular.

  Of course all the above discussion is oriented towards the case when $\mathfrak p \subset \mathfrak g$, with $\mathfrak g$ simple of general type.

  \subsection{Factorisation}\label{1.10}

  In the notation of \ref {1.2} suppose that $C_1,C_k$ are neighbouring columns of height $s$.  Then the weight of the Benlolo-Sanderson invariant is the Kostant weight $\varpi_s-w_0\varpi_s$, \cite [VIII]{BS}, \cite [2.20]{JM}, as is obvious from its description in \cite [3.6.3]{FJ}.  In the Bourbaki notation \cite {Bo}, set $\varepsilon_{u,v}=\sum_{i=u}^v \varepsilon_i$. Then the above Kostant weight takes the form $\varepsilon_{1,s}-\varepsilon_{n-s+1,n}$.  Now suppose that between $C_1,C_k$, there is a further column $C''$ of height $s$ and let $r$ denote the sum of the column heights up to $C''$.  Then the weight of the Benlolo-Sanderson invariant defined by the pair $C_1,C''$ (resp. $C'',C_k$) is just $\varepsilon_{1,s}-\varepsilon_{r-s+1,r}$ (resp. $\varepsilon_{r-s+1,r}-\varepsilon_{n-s+1,n}$).  This gives the following conclusion (which is otherwise not obvious)

  \begin {lemma} The Benlolo-Sanderson invariant for the pair $C_1,C_k$ is the product of the Benlolo-Sanderson invariants for the pairs $C_1,C''$ and $C'',C_k$.
     \end {lemma}
  \begin {proof}  Since the invariant ring is a domain, the product is non-zero.  By the above computation it has the same weight as the Benlolo-Sanderson invariant for the pair $C_1,C_k$.  By Richardson's theorem which implies multplicity $\leq 1$, these espression must coincide up to a non-zero scalar.
  \end {proof}
  \textbf{Remark.}  Obviously this factorisation also holds for any triple of neighbouring columns.

  \section {Towards the description of the irreducible components of $\mathscr N$ in type $A$.}\label {2}

\subsection{Matrix Co-ordinates}\label{2.1}

From now on we assume that $G$ is simple of type $A$, that is isomorphic to $SL(n)$ for some integer $n>1$.  In this case the Weyl group $W$ is just the symmetric group $S_n$ on $n$ letters.

Let $\textbf{M}_n$, or simply $\textbf{M}$, denote the set of $n \times n$ matrices and as before (\ref {1.6}) we write $x_{i,j}:i,j \in [1,n]$, for the standard matrix units.

Then the  Lie algebra $\mathfrak {sl}(n)$ of $SL(n)$ is just the subspace $\textbf{M}_n$ of  traceless matrices given a Lie bracket through the commutator.

Let $\alpha_i:i \in [1,n-1]$ be the simple roots in the Bourbaki notation \cite {Bo} and set $\alpha_{i,j}:=\alpha_i+\cdots+\alpha_{j-1}$.\, for $i<j$.  Observe that $x_{i,j}$ is the root vector $x_{\alpha_{i,j}}\in \mathfrak n$.

\subsection{Column and Row Labelling}\label{2.2}

We define a diagram $\mathscr D$ to be a set of columns $\{C_i\}_{i=1}^k$ labelled from left to right with the integers $1,2,\ldots,k$.

\

N.B.  We shall not always label the columns in this precise same fashion.  Indeed we may fix $s \in \mathbb N$ and label (sequentially from left to right) only the columns of height $\geq s$ and even then only between a fixed pair of neighbouring columns of height $s$ as in \ref {2.4}.  Again even amongst the columns of height $\geq s$ we may omit some labels as in \ref {5.4.5}.

\

Let $c_i$ denote the height of $C_i:i=1,2,\ldots,k$.

 Let $\{R_j\}_{j=1}^\ell$ be the rows of $\mathscr D$ labelled with positive integers increasing from top to bottom.  For all $i \in \mathbb N^+$, let $R^i$ denote the union of the rows $\{R_j\}_{j=1}^i$.


  \subsection{Blocks and Column Blocks}\label{2.3}

  Let $\pi'$ be a subset of $\pi$ and recall that $\mathfrak p_{\pi'}$ is the parabolic subalgebra containing $\mathfrak b$ defined by $\pi'$.  Let $W_{\pi'}$ be the subgroup of $W$ generated by the simple reflections $s_\alpha:\alpha \in \pi'$ and $w_{\pi'}$ its unique longest element.

Recall that the Levi factor $ L_{\pi'}$ of $P_{\pi'}$ is given by a set of blocks $\textbf{B}_i$ on the diagonal of sizes $c_1,c_2,\ldots,c_k$, where the $c_i-1$ are the cardinalities of the connected components of $\pi'$ and $\sum_{i=1}^k c_i=n$.  Its Lie algebra $\mathfrak l_{\pi'}$ is complemented in $\mathfrak p_{\pi'}$ by the nilradical $\mathfrak m_{\pi'}$ of $\mathfrak p_{\pi'}$.

We view the parabolic $P_{\pi'}$ as being specified by the composition $(c_1,c_2,\ldots,c_k)$.  This convention will be used throughout.

As in \ref {1.2}, let $\mathscr D_\mathfrak m$ be the diagram with columns $C_i$ of heights $c_i:i=1,2,\ldots,k$ defined by the block sizes of $L_{\pi'}$.


 For each column $C_i$, let $\textbf{C}_i$ denote the rectangular block in $\mathfrak m$ lying above $\textbf{B}_i$.  We call it the $i^{th}$ column block. Its width is $c_i$ and its height $\sum_{j < i}c_j$.

In this presentation one has $\textbf{C}_1=0$ and $\mathfrak m = \oplus_{i=2}^k \textbf{C}_i$.


Let $\mathscr T$ be a tableau obtained from $\mathscr D$ in which the integers increase down the columns and along the rows from left to right. It need not be a standard tableau because the $c_i$ need not be increasing - in other words there may be gaps in the rows of $\mathscr T$.

Notice that $\mathscr T_\mathfrak m$ of \ref {1.6} obtains from $\mathscr D_\mathfrak m$ by this procedure.

Yet $\mathscr T$ becomes a standard tableau when we shift boxes in a given row from right to left to close up gaps. Indeed the content of each row does not change and in the new tableau entries still increase down the columns.  Then through the Robinson-Schensted correspondence and a result of Steinberg \cite {St}, $\mathscr T$ corresponds to an orbital variety $\mathscr V_\mathscr T$, for $\mathfrak {sl}(n)$ and every orbital variety is so obtained exactly once.

One may recover $\mathscr D$ from $\mathscr T$ by forgetting the entries.  It is called the shape of the tableau.  The set $c_1,c_2,\ldots,c_k$ of column heights determine the dimension of $\mathscr V_\mathscr T$ through the formula
$$2\dim \mathscr V_\mathscr T=n(n-1)-\sum_{i=1}^k c_i(c_i-1). \eqno{(*)}$$

\subsection{Column Shifting}\label{2.4}

To a pair of neighbouring columns $C,C'$ of height $s\geq 1$ in $\mathscr T_\mathfrak m$, a new tableau, associated to a hypersurface orbital variety for $\mathfrak {sl}(n)$ contained in $\mathfrak m$, is constructed as follows.

The construction follows that in \cite [2.15]{JM}, but has a crucial difference.

It is a first instance in which columns are labelled slightly differently than \ref {2.2}.

Thus let $C_j:j \in [1,u]$ denote the columns of height $>s$ between $C,C'$ and set $C_0=C',C_{u+1}=C$.

For all $i \in [1,u]$, let $C^{\leq s}_i$ (resp. $C^{>s}_i$) denote that part of $C_i$ lying on and above (resp. strictly below) row $R_s$.

\

 \textbf{N.B.}   Here there is room for misunderstanding, as we are using the British (rather than the French) convention of starting the columns in row $R_1$ and going downwards.  Thus one might have called $c_i$ the depth, rather than the height of $C_i$.  Then despite the notation $C^{\leq s}_i$ is the \textit{upper} part of $C_i$.  On the other hand the indices in $C^{\leq s}$ are smaller than those in $C^{>s}_i$ - see Figure $1$.

 \

  For $i \in [1,u-1]$, let $C_i'$ denote the column in which $C_i^{>s}$ has been replaced by $C_{i+1}^{>s}$. Let $C_{u+1}'$ (resp. $C_u'$) be the column with just the box containing $n$ removed from row $R_s$ (resp. placed in row $R_{s+1}$). Finally let $C_0'$ be the column with $C_1^{>s}$ adjoined below $C_0$.

  This manipulation is exemplified in Figure $1$ with the result described in Figure $2$.

  Combined with the columns of height $<s$ which are left in their original position, this defines the tableau $\mathscr T_\mathfrak m(n)$.

  More picturesquely $\mathscr T_\mathfrak m(n)$ is obtained from $\mathscr T_\mathfrak m$, by shifting $n\in C'$ into the $R_{s+1}\cap C_u$ and then, ignoring columns of height $<s$ between $C,C'$, shifting those parts of the columns strictly above $R_s$ by one position to the left.

  Notice the \textit{shape} $\mathscr D_{\mathfrak m}(n)$ of  $\mathscr T_{\mathfrak m}(n)$ is obtained from the shape $\mathscr D_{\mathfrak m}$  of  $\mathscr T_{\mathfrak m}$, by replacing the two columns ($C,C'$) of height $s$ by columns $(C'_u,C'_{u+1})$ of heights $s+1,s-1$, and permuting the heights of the remaining columns. Then by \ref{2.3}$(*)$  the dimension of $\mathscr V_{\mathscr T_{\mathfrak m}(n)}$ is one less than $\dim \mathfrak m$.  Hence
   \begin {lemma}  $\mathscr T_{\mathfrak m}(n)$ represents a hypersurface orbital variety in $\mathfrak m$.
   \end{lemma}


Let $t_\mathfrak m(n)$,  denote the largest entry in the first column of $\mathscr T_{\mathfrak m}(n)$, equivalently the largest entry of $C_1$.  One may remark that $t_\mathfrak m(n)=n$, if there are no columns of height $\geq s$ strictly between $C,C'$.

\

\textbf{N.B.}  $\mathscr T_{\mathfrak m}(n)$ is not the tableau obtained in \cite [2.17]{JM}.  For the latter we must further move all the boxes as far as possible to the left to close up the gaps in the rows.  By (Prop. \ref {2.6} combined with \cite [2.17]{JM})  these tableaux represent the same orbital variety but are described by different elements of the Weyl group, the latter not being uniquely determined by the orbital variety.

\subsection{Weyl Group Elements}\label{2.5}

To a tableau $\mathscr T$ we associate an element $w_c(\mathscr T)\in W$, following closely \cite [2.8]{JM}.  As noted in \cite [Lemma 2.8]{JM} this element comes from the Robinson-Schensted correspondence.

The construction of $w_c(\mathscr T)$ is as follows.  Let $C_1,C_2,\ldots,C_k$ be the columns of $\mathscr T$ with their entries inserted and viewed as a sequence of integers.  Given a sequence $S$ of integers, let $\hat{S}$ denote the sequence taken in the opposite order.  Then  $w_c(\mathscr T)$ is given as an element of $S_n$ through the word form $[\hat{C}_1,\hat{C}_2,\ldots,\hat{C}_k]$.

The orbital variety closure defined by $\mathscr T$, through the Robinson-Schensted
correspondence, is just $\overline{B.(\mathfrak n \cap w_{c(\mathscr T)}\mathfrak n)}$.

We remark that in particular  $w_c(\mathscr T_{\mathfrak m_{\pi'}}) =w_{\pi'}$.

Different choices of $\mathscr T$ may given the same orbital variety closure.  The possible choices are determined by the Robinson-Schensted correspondence.  Here we believe we are making a better choice than in \cite [2.18]{JM} by not closing up gaps in the rows of a tableau.  This improvement manifests itself in Prop. \ref {6.9.8}.



For all $w \in W$, set $S(w):=\{\alpha \in \Delta^+| w\alpha \in \Delta^-\}$ and $\mathfrak u:= \mathfrak n \cap w(\mathfrak n)=\sum_{\alpha \in \Delta^+ \setminus S(w^{-1})}\mathbb C x_{\alpha}$.

The root spaces in $\mathfrak n$ which do not lie in $\mathfrak n \cap w(\mathfrak n)$ are the $x_\alpha : \alpha \in S(w^{-1})$.  These are determined by the following well-known result, see for example \cite [Lemma 2.3]{JM}, which we reproduce below

\begin {lemma} Suppose $i<j$.  Then $\alpha_{i,j} \in S(w^{-1})$ if and only if $i$ is after $j$ in the word form of $w$.
\end {lemma}

\subsection{The Benlolo-Sanderson Invariant and its Zero Set}\label{2.6}

 In this section, $C,C'$ are neighbouring columns of height $s$.  In describing the associated Benlolo-Sanderson invariant we can ignore all columns not between this pair. Thus we can take $C=C_1 ,C'=C_k$ and revert to the labelling of the columns given by \ref {2.2}.

Following \cite [4.1.6]{FJ} we describe the Benlolo-Sanderson invariant in the following way.

Let $M_s$ be the $n-s \times n-s$ minor in the bottom left hand corner.  We consider $M_s$ as a polynomial function on $\mathfrak m+\Id$, through the Killing form. The leading term, $\gr M_s$, of its restriction is the Benlolo-Sanderson invariant.   As noted in \cite {BS} its degree $d(\gr M_s)$ is given by
$$d(\gr M_s)= \sum_{i=1}^{k-1} \min (c_i,s). \eqno{(*)}$$


As noted in \ref {2.4}, $\mathscr T_{\mathfrak m}(n)$ corresponds to a hypersurface orbital variety in $\mathfrak m$. This is made more precise by the following.

\begin {prop} Set $w=w_c(\mathscr T_{\mathfrak m}(n))$. The closure of $B.(\mathfrak n \cap w(\mathfrak n))$ is the zero variety of $\gr M_s$ in $\mathfrak m$.
\end {prop}

\begin{proof}  The argument is not dissimilar to the proof of \cite [Prop. 3.8]{JM}, except that we do not use Melnikov theory.

As in \ref {2.5} set $\mathfrak u:= \mathfrak n \cap w(\mathfrak n)$.  By Lemma \ref {2.5} and the construction, one obtains $\mathfrak u \subset \mathfrak m $.

Let $d_\mathscr D$ denote the total number of boxes in $\mathscr D$ between $C_1,C_k$ on rows $R_{s'}:s'>s$ and $m= \sum_{i=1}^k c_i$.  One easily checks that the right hand side of $(*)$ is just $m-s-d_\mathscr D$.  Moreover as already noted in \cite [Lemma 2]{BS} the evaluation of $M_s$ on $\mathfrak m+a\Id$ is divisible by $a^{d_\mathscr D}$ and the value of $gr M_s$ is obtained by dividing the evaluation of $M_s$ on $\mathfrak m + a\Id$  by $a^{d_\mathscr D}$ and setting $a=0$.  As in the proof of \cite [Prop. 3.8]{JM}, this will become obvious from the computation below (as long as one knows where one is going).

As in \cite [Prop. 3.8]{JM}, the easiest case is when $\mathscr T_{\mathfrak m}$ has height $s$, in other words $c_i<c_1=c_k=:s$ for all $i$ with $1<i<k$ in the labelling of \ref {2.2}. In this case $d_\mathscr D=0$ and $\gr M_s$ is obtained by evaluation of $M_s$ on $\mathfrak m$. Moreover $\mathfrak u$ is just $\mathfrak m$ with its last column removed.  In this case $M_s$ obviously vanishes on $\mathfrak u$.

We establish the general case by induction on the number of columns of height $>s$ between the neighbouring columns $C_1,C_k$ of height $s$. If this number is zero, then we are in the previous case settled above.

Let us call the root vectors in $\mathfrak m$ not lying in $\mathfrak u$, the excluded root vectors.  Each lie on the intersection of a column and row of $\textbf{M}_n$.  In the matrix presentation we describe them by a $0$ entry, because this entry is set equal to zero.  In \ref {6.9.3} we denote then by a $O$ entry since it may encircle another entry.

 Let $C'$ be the leftmost column of $\mathscr D_\mathfrak m$ of height $r+s>s $ strictly to the right of $C_1$.


Recall the notation of \ref {2.4} and set $t=t_\mathfrak m(n)$. It is just the largest entry in $C'$.  Then $t-j: j \in [0,r-1]$ are the remaining entries in $C'^{>s}$.

Take $v \in t-j: j \in [0,r-1]$.
In the word form of $w=w_c(\mathscr T_{\mathfrak m}(n))$, every positive integer $<v$ appears after $v$.  Then by Lemma \ref {2.5}, the entire column in $\textbf{M}_n$ strictly above the $(v,v)$ entry is excluded from $\mathfrak u$.  This means that evaluation of $M_s$ on $\mathfrak u$, sets all the entries above the $(v,v)$ entry equal to $0$. On the other hand the $(v,v)$ entry itself is $a$.

We conclude that the evaluation of $M_s$ on $\mathfrak u$ is $a^r$ times the evaluation of $M_s$ on the corresponding matrix in which $C'$ has been replaced by a column $C''$ of height $s$.  (If one likes one can shift the indices on columns to its right down by $r$ but this is of no particular importance.)   One may note that there are no excluded roots in the columns to the left of $C'$ and the excluded roots to the right remain unchanged, up to this renumbering. All this is an easy consequence of the way $w$ is defined.  Simply the displacements described in \ref {2.4} concerning the columns strictly to the right of $C'$ are unchanged by a change in the height of $C'$.

Now our new diagram has three columns of height $s$, namely $C_1,C'',C_k$.  By Lemma \ref {1.10} the Benlolo-Sanderson invariant $\gr M_s$ is a product of the corresponding Benlolo-Sanderson invariants for the two neighbouring pairs.  The left hand factor need not vanish on $\mathfrak u$ but the right hand factor corresponding to the columns between $C'',C_k$ vanishes on $\mathfrak u$ because the excluded roots have not changed between these columns whilst $d_\mathscr D$ has been decreased by $r>0$, so the induction hypothesis applies.

We conclude that $\gr M_s$ vanishes on $\mathfrak u$ and since it is $P$ semi-invariant, it also vanishes on $B.\mathfrak u$, so on its closure.

Finally by Lemma \ref {2.4}, $B.\mathfrak u$ has codimension $1$ in $\mathfrak m$.  Consequently the zero variety of $\gr M_s$ is precisely the closure of the irreducible set $B.\mathfrak u$.

\end {proof}

\subsection{The Excluded Roots}\label{2.7}

Define $\mathfrak u$ as in \ref {2.5}.  It and its complement in $\mathfrak n$ is spanned by root vectors.  The latter are called the excluded root vectors, or excluded co-ordinates, and their roots, the excluded roots.  These were partly determined in \ref {2.6}. Here we determine them all explicitly.

\

First recall how we have labeled the columns $C_i:i \in [0,u+1]$ in \ref {2.4} in $\mathscr T_\mathfrak m$.  In particular $C_0,C_{u+1}$ are neighbouring columns of height $s$, whilst $C_i:i \in [1,u]$ are of height $>s$.  Columns outside this pair and columns of height $<s$ within this pair are \textit{left unlabelled}.

  By Lemma \ref {2.5} the excluded roots are those whose roots lie in $S(w^{-1})$.

 For all $j \in [1,u]$, let $t_j$ be the largest entry in $C_j$. Let
 $C_j^i$ be the $i^{th}$ column of  $\textbf{C}_j$ viewed as a column in $\textbf{M}$.  Choose $\ell_j \in \mathbb N^+$, so that $\textbf{C}_j^{\ell_j}$ is the rightmost column of  $\textbf{C}_j$.

 The excluded co-ordinates of $\mathfrak m$ lie in $\textbf{C}_j^{\ell_j+1-i}:i \in [1,c_j-s]:j \in [1,u]$ and in the last column of $\textbf{M}$.  They are given below.

  First with respect to above labelling the excluded co-ordinates do not depend on the choice of $i \in [1,c_j-s]:j \in [1,u]$, so we can just describe which elements in $\textbf{C}_j^{\ell_j}:j \in [1,u]$ are excluded.

  Recall that the Levi factor of $\mathfrak p$ is given by  $c_i\times c_i$ blocks $\textbf{B}_i$ down the diagonal of  $\textbf{M}$ defined by the heights $c_i$ of the columns $C_i:i \in [1,u]$ of $\mathscr T_\mathfrak m$.

  Fix $j \in [1,u+1]$.  The excluded co-ordinates in $\textbf{C}_j^{\ell_j}$ form a single column starting from the highest row of $\textbf{B}_{j-1}$ down to lowest row of $\textbf{C}_j$ (which is row just above the highest row of $\textbf{B}_j$) with a gap between the $(s+1)^{th}$ and $c_{j-1}^{th}$ row of $\textbf{B}_{j-1}$. Notice that if $j=1$, then $c_{j-1}=s$ and so there is no gap. (This was true for the special case considered in \ref {2.6}.) One may remark that the excluded co-ordinates below this gap all correspond to columns of height $<s$ between $C_{j-1},C_j$.

  The excluded co-ordinates in the last column of $\textbf{M}$, that is to say in $\textbf{C}_{u+1}^s$, follow the rule given above.  There are no multiple columns of excluded roots in the rightmost column block $\textbf{C}_{u+1}$, because only one element namely $n$ is shifted to the left in the construction of \ref {2.4}.

   \

\textbf{Example 1.}  Take the composition $(2,1,3,1,4,1,2)$.  Let us just compute the excluded co-ordinates arising from the neighbouring columns of height $2$.  The first set occurs in column $6$ in rows $(1,2,3)$.  The second set occurs in columns $10,11$ in rows $4,5,7$ illustrating the gap and there being two columns, since $c_5-s=2$.  The final set occurs in column $14$ (only) in rows $8,9,12$.
One may remark that in this case $\mathfrak u$ is stable under the action of the space $\mathfrak l^-$ spanned by the negative root vectors of $\mathfrak p$.  This is a general fact.  Thus $B.\mathfrak u$ is $P$ stable, as is to be expected.

\

\textbf{Example 2.}  Take the composition $(1,3,2,1)$.  Then $w=4316275$ in word presentation.  Then the co-ordinates in $\mathfrak m$ which are excluded form the set \newline $\{(1,3),(1,4),(2,6),(5,7)\}$.  On the other hand the recipe of \ref {5.4} gives a Weierstrass section with $e=x_{1,2}+x_{2,5}+x_{3.6}+x_{6,7}$ and $V =x_{5,7}$.  In this the co-ordinates which make up $e$ (resp. $V$) do not lie (resp. do lie) in the excluded set.  We show (Cor. \ref {6.9.8}) that this holds in all cases, which a priori is quite \textit{extraordinary} and as we shall see has an important interpretation.

\

Extending partly our previous notation we let $L$ (resp. $L^-$) denote the connected subgroup of $G$ with Lie algebra $\mathfrak h$ together with root vectors $x_\alpha$ with $\alpha \in \Delta \cap \mathbb Z\pi'$ (resp.  $\alpha \in \Delta \cap -\mathbb N\pi')$.

\begin {lemma}  $\mathfrak u:=\mathfrak n \cap w(\mathfrak n)$ is stable under the adjoint action of $L^-$.  In particular $B.\mathfrak u$ and its closure are $P$ stable.
\end {lemma}

\begin {proof}  The first statement holds because the excluded elements lie in the last $c_j-s$ columns of each column block $\textbf{C}_j:j \in [1,u+1]$ and in the last column of $\textbf{M}$ and lie in an unbroken set of rows down to a row determined by the previous column blocks as given above.  The second statement holds by induction of the number of columns.  The last part is clear.
\end {proof}

\subsection{The Key $B$ Saturation Set}\label{2.8}

For any pair of neighbouring columns, equivalently for any hypersurface orbital variety in the nilradical $\mathfrak m_{\pi'}$ of a parabolic $\mathfrak p_{\pi'}$, we obtain a subspace $\mathfrak u_i:=\mathfrak n \cap w_i \mathfrak n$.  Here $w_i$ is the \textit{very specific } element of $W$ constructed above.

Set $\mathfrak u_{\pi'} = \cap_{i=1}^k \mathfrak u_i$. Since the nilfibre  $\mathscr N$ is just the intersection of the $\overline{B.\mathfrak u_i}$, we obtain $\overline {B.\mathfrak u_{\pi'}} \subset \mathscr N$.

Let $g$ denote the number of pairs of neighbouring columns of $\mathscr D_\mathfrak m$. By \ref {1.2} the number of generators of the polynomial algebra $\mathbb C[\mathfrak m]^{P'}$ is just $g$.

It is clear that $\mathfrak u_{\pi'}$ is a subalgebra of $\mathfrak m$.  If it is complemented in $\mathfrak n$ by a subspace spanned by the root vectors, then $\overline{B.\mathfrak u_{\pi'}}$ is an orbital variety closure.  In this case we obtain from Robinson-Schensted theory that $\dim \overline{B.\mathfrak u_{\pi'}}=\dim \mathfrak m -g$.

It is easy to check that $ \mathfrak u_{\pi'}$ is complemented in $\mathfrak n$ by a subalgebra if the neighbouring columns form a nested sequence of strictly decreasing height, for example if $\pi'$ is defined by the composition $(3,2,1,1,2,3)$.  This fails for $(2,1,1,1,2)$ and for $(1,2,2,1)$. In the first case $\overline{B.\mathfrak u_{\pi'}}$ is not an orbital variety closure, in the second case it is, as noted in \ref {6.10.9}, Examples $2$, $3$.

A main result of this paper is the very remarkable fact (Cor. \ref {6.9.8}) that $\overline {B.\mathfrak u_{\pi'}}$ is an irreducible component of $\mathscr N$ of dimension $\dim \mathfrak m-g$.

\subsection{More on Excluded Roots}\label{2.9}

Let $C_1,C_2,\ldots,C_k$ be the columns of $\mathscr T_{\mathfrak m}$.  Let $\mathscr T'_{\mathfrak m}$ be the tableau obtained from $\mathscr T_{\mathfrak m}$ by omitting the last column $C_k$ and set $\mathfrak u'=\cap_{i=1}^{k-1}  \mathfrak u_i$, that is to say omitting the contribution from the last column.

\begin {lemma}

$(i)$. There are no excluded roots in the last column block $\textbf{C}_k$ implied by $\mathfrak u'$.

\

$(ii)$.  If there is no column neighbouring to $C_k$, there are no excluded roots in the last column block $\textbf{C}_k$ implied by $\mathfrak u_{\pi'}$.
\end {lemma}

\begin {proof}  By the given numbering in $\mathscr T_{\mathscr m}$ the entries of $C_k$ exceed all those in previous columns.  Hence (i).  Under the hypothesis of (ii), no entries of $C_k$ are moved. Thus the assertion follows by Lemma \ref {2.5} and (i).
\end {proof}

\section{A possible recipe to construct a Weierstass section in general}\label{3}

In this section we return to the assumption that $P$ is a parabolic subgroup of a connected simple algebraic group acting on the Lie algebra $\mathfrak m$ of its nilradical.

Recall that $\dim V=g$ is the number of generators of $C[\mathfrak m]^{P'}$.  Set $d = \dim \mathfrak m -g$.

Throughout the remainder of this paper, given $e \in \mathscr N$, we simply say that $e$ is regular, if it is regular in $\mathscr N$ (See  \ref {1.5}).  In an ideal world\footnote{Tout est pour le mieux dans le meilleur des mondes possibles - Professor Pangloss \cite{A}.}, $\mathscr N$ is equidimensional of dimension $d$ and this will just mean that $P.e$ had dimension $d$. In fact we shall see that in type $A$ be combining \ref {6.9.4}, \ref {6.9.5} and \ref {6.9.8}, it follows  that the irreducible  component $\mathscr C$ of $\mathscr N$, containing $e$ constructed to be part of the Weierstrass section $e+V$ in \ref {5.4}, has in fact dimension $d$.  However it can happen \ref {6.10.7} that $P$ has no dense orbit in $\mathscr C$.

 Let $\mathscr N_{reg}$ denote the set of regular elements of $\mathscr N$.

\subsection{A Null Intersection Lemma}\label{3.1}

Suppose that $e+V$ is a Weierstrass section for the action $P'$ on $\mathfrak m$, such that $e \in \mathscr N$.  In this situation we obtain the following

\begin {lemma}

\

(i) $ P.(e+V) \cap \mathscr N= P.e$.

\

(ii)  $\mathfrak p.e \cap V =\{0\}$.

\

(iii) $\dim P.e = \dim \mathfrak p.e \leq \dim \mathfrak m -\dim V$.

\end {lemma}

\begin {proof}  Given $v \in V$ non-zero, let $V'$ be a complement to $\mathbb Cv$ in $V$ and choose $\xi \in V^*$ to vanish on $V'$ and to be non-zero on $v$.

By the surjectivity in the definition of a Weierstrass section, there exists a homogeneous $P'$ invariant polynomial $f$ of positive degree whose restriction to $e+V$ is $\xi$.

Since $\mathscr N$ is $P$ stable, to prove (i) it is enough to show that $e+v \in \mathscr N : v \in V$ implies $v=0$.  If $v \neq 0$, define $f$ as above. Then $0=f(e+v)=\xi(v) \neq 0$, which is a contradiction. Hence (i).

If (ii) fails, there exist $v \in V$ non-zero and $x \in \mathfrak p$ such that $x.e=v$.  Choose $f$ as above.  Then for all $c \in \mathbb C$, one has $f(e +cv)=\xi(cv)=c\xi(v)$.

On the other hand $\exp cx \in P$, so $\exp cx.e \in \mathscr N$. Hence $0=f(\exp cx).e)=f(e+cv) +c^2\mathfrak m[c])$. Consequently
$$0=(\partial f(\exp cx)/\partial c)_{c=0}=(\partial f(e+cv)/\partial c)_{c=0}= \xi(v)\neq 0,$$
which is a contradiction.  Hence (ii).

The equality in (iii) is standard, the inequality follows from (ii).

\end {proof}

\subsection{An Adapted Pair from Regularity}\label{3.2}

Let $e+V$ be a Weierstrass section as in \ref {3.1}.

%
%
%
%
%
%

Let $\mathscr C$ be an irreducible component of $\mathscr N$.  Suppose that $\mathscr C$ has dimension $d$.


\begin {lemma}  There exists $h \in \mathfrak h$ such that $h.e=-e$, that is $(e,h)$ is an adapted pair.
\end {lemma}

\begin {proof} Through its definition $\mathscr N$ is a cone. Since  $\mathscr C$ is an irreducible component of $\mathscr N$, it is also  a cone. On the other hand (by hypothesis) it admits a dense $P$ orbit $P.e$. Yet for all $c \in \mathbb C$ non-zero, $P.ce$ is also a dense $P$ orbit in $\mathscr C$, so the two must coincide. In particular $ce \in P.e$. Yet $P$ is an algebraic group so the last relation implies that it admits a one-dimensional torus $T$ such that $ce \in T.e$ and so $T.e=(\mathbb C\setminus \{0\})e$.
Consequently one can find a semisimple element $h$ in the Lie algebra of $T$ such that $h.e=-e$.   Up to (simultaneous) conjugation of $h,e$ one can take $h \in \mathfrak h$.

\end {proof}

\subsection{An Eigenvalue Criterion for Surjectivity}\label{3.3}

Suppose $e \in \mathscr N$ with $(e,h)$ an adapted pair,

Then one can choose $V$ to be an $h$ stable complement to $\mathfrak p.e$ in $\mathfrak m$.  As noted for example in \cite [5.2.5]{FJ} a standard argument shows that $P.(e+V)$ is dense in $\mathfrak m$.  Consequently the homomorphism $\varphi$ from the invariant algebra $\mathbb C[\mathfrak m]^{P'}$ to $\mathbb C[e+V]$ defined by restriction of functions is injective.


As in say \cite [2.2.4] {FMJ}, a simple trick allows one to show that $V$ can be chosen to have a basis of root vectors (in $\mathfrak m$) with linearly independent roots.

More is needed (and necessary) to obtain surjectivity. We can choose the generators $\{p_i\}_{i=1}^g$ to be homogeneous of degree say $d_i$ with $h$ eigenvalue $\lambda_{p_i}$. Then its image $v_i=\varphi(p_i)$  has $h$ eigenvalue $\lambda_{v_i}$ with
$$ \lambda_{v_i} - \lambda_{p_i} = d_i-1, \forall i \in [1,g].$$

Of course the $v_i$ are polynomials in the basis elements $\{u_i\}_{i=1}^g$ which we can assume to be $h$ eigenvectors of eigenvalue $\lambda_{u_i}$.  Given that the latter are non-negative then surjectivity will follow from the sum rule
$$ \sum_{i=1}^g (\lambda_{u_i} - \lambda_{p_i} - (d_i-1))=0.$$

Unlike the coadjoint case \cite [Sect. 6]{JS}, this does not appear to hold automatically.

\subsection{}\label{3.4}

We shall see in \ref {5.4} that $e$ is a sum of root vectors whose roots are linearly independent.   However in contrast to $V$ (see \ref {3.3}) there is no easy way to make this so, or to be more precise there is no way, easy or otherwise, to make $e \in \mathscr N_{reg}$, a sum of root vectors with linearly independent roots.  This does hold for coadjoint action of a simple group $G$ on $\mathfrak g^*$; but already the proof is highly involved using the classification of distinguished nilpotent elements \cite {CE}.

That we can choose $V$ as in the last part of \ref {3.3} to obtain surjectivity is obtained here in type $A$ and for our special choice of $e$, by hard combinatorics which may turn the stomachs of some readers\footnote {Unlike Shakespeare, Henry V, Act II, Prologue
- We’ll not offend one stomach with our play.}.
\section{Attempts at Regularity in type $A$.}\label{4}

Here we return to the assumption that $P$ is a parabolic subgroup of $SL(n)$.

\subsection{The Mechanics of Gating}\label{4.1}

In \cite [Sect. 4]{FJ} we showed that the action of $P'$ on $\mathfrak m$ admits a Weierstrass section $e+V$.

Very briefly our construction, introduced in \cite [4.2]{FJ}, runs as follows.   As noted in \ref {2.3}, a line joining boxes (specifically with labels $i,j$)  defines a root vector (specifically $x_{i,j}$) in $\mathfrak m$.  It is denoted by $\ell_{i,j}$, or simply by $(i,j)$.

In the above $\ell_{i,j}$ is given either a label $1$ or a label $*$.  (The use of $\ast$ instead of a $0$, was  recommended to us by Ilya Zakharevich, during our lecture.)

%

\

Then $e$ is defined to be the sum of the root vectors given by the lines labelled by $1$ and $V$ the direct sum of the root subspaces given by the lines labelled by $\ast$.

In \cite [Thm. 6.2]{FJ}, we found a labelling to make $e+V$ is a Weierstrass section.

It was further found \cite [5.4.2]{FJ} that $e \in \mathscr N$. \textit{Yet it is not always the case that $e$ is regular.}  For example suppose that $\mathfrak p$ is defined by blocks of sizes which form the sequence $(1,2,2,1)$, then our recipe gives $V=\mathbb Cx_{4,6}+\mathbb Cx_{3,5}$ and $e=x_{1,2}+x_{2,4}$, which is not regular. A natural modification would be to ``gate'' the lines carrying $*$ in the sense of \cite [4.4]{FJ}.  This process which can be considered as removing these lines, and then to join boxes with lines carrying a $1$ which are ``freed'', by this gating. Here the most natural solution is to add $x_{3,6}$ to the expression for $e$ above.

Actually this still does not make $e$ regular.  Yet here the box carrying $5$ has no right going line, so we can instead adjoin $x_{5,6}$ which does make $e$ regular. The matter becomes even more delicate when the box containing $5$ already has a right going line due to the presence of a column of height $3$ to the right as in the case of the parabolic defined by the block sizes $(3,1,2,2,1,3)$ which seemed at first to show that we cannot choose $e$ to be regular.  In fact this example does admit a Weierstrass section $e+V$ with $e \in \mathscr N_{reg}$ by taking $V=\mathbb Cx_{7,9}+\mathbb Cx_{6,8}+\mathbb Cx_{6,12}$ and $e=x_{1,4}+ x_{4,5}+ x_{5,7}+ x_{9,10}+ x_{2,6}+ x_{8,9}+x_{3,8}+x_{7,11}$.


\subsection{Maximal Gating}\label{4.2}

This last and rather simple example should convince the reader that finding a Weierstrass section $e+V$ with $e \in \mathscr N_{reg}$, is a rather delicate business.

The first thing to recognize is that non-regularity results from having at the end of our procedure more than one ungated $*$. The simplest case of the Borel itself when the block sizes are all one. In this case we could have chosen $V$ to be the sum of the simple root vectors and $e=0$ which is obviously not regular (except for $\mathfrak {sl}(2)$).  An appropriate choice of $e$ in this case is obtained by joining alternate boxes with a line carrying an $*$, giving $e = \sum_{i=1}^{n-2}x_{i,i+2}$.  However in general regularity may still not result as in the second of the examples in \ref {4.1}.

We must at least modify ``step $3$'' of \cite [4.4]{FJ} so that at each induction stage, there is at most one ungated $*$. We call this a ``maximal gating'' rather than the ``minimal gating'' used in \cite [Sect. 4]{FJ}.

Notably it is more natural.  In particular it behaves well with respect to adjoining columns on the right (Lemma \ref {5.4.9}), but not so well for adjoining columns on the left (Lemma \ref {5.4.10}).

Again recall the notation of \ref {2.8}. When $e$ is given by a maximal gating, we show that it belongs to $\mathfrak u_{\pi'}$ and use this fact to show that
$\overline{B.\mathfrak u_{\pi'}}$ is an irreducible component of $\mathscr N$ of dimension $\dim \mathfrak m -g$.  It is quite remarkable that this can be achieved.

\subsection{Possible Failure of Regularity}\label{4.3}

It turns out that even with a maximal gating, the resulting element $e$ need not be regular. Again the augmented element $e_{VS}$ described in {1.7} need not be regular.

Indeed recall the result noted in last paragraph of \ref {4.2}. We use this in \ref  {6.10.7} to establish the second main result of this paper  (\ref {6.10.7}), namely  that the irreducible component $\overline{B.\mathfrak u_{\pi'}}$ of $\mathscr N$ containing $e$ need not admit a dense $P$ orbit.

\section{Modifying Step $3$ of \cite [4.4] {FJ}}\label{5}

\subsection {Minimal versus Maximal Gating}\label {5.1}

Here we shall modify the construction of \cite [4.4]{FJ} in which a ``minimal'' gating was used, by taking a ``maximal'' gating.  Otherwise the construction is similar to \cite [4.4]{FJ}.  Yet the details and the result may look \textit{dramatically} different as noted in the example below.  The comparison is made in Figures $3,4,5$.

Consider the parabolic defined by the sequence $(3,1,2,2,1,3)$, with $s=2$.

After step $2$, the lines $\ell_{7,9},\ell_{6.8},\ell_{3,12}$ carry a $\ast$.

 In the minimal gating, $\ell_{7,9}$ is \textit{not} gated at stage $2$.  In particular the lines in $R_i;i=1,2$ are not changed from step $2$.  Then in stage $3$ both the lines carrying a $*$ in $R^2$ are gated.  Finally $3,8$ is joined by a line carrying a $1$ and $7,12$ by a line carrying a $*$ following the first part of the third from bottom paragraph of \cite [4.4.10]{FJ} and $6,9$ are joined by a line carrying a $1$ following the second part of the cited paragraph.

 In the maximal gating the line $\ell_{7,9}$ \textit{is} gated at stage $2$ and the pairs $8,9$ and $7,11$ are joined by a line following a procedure similar to the first part of the third from bottom paragraph of \cite [4.4.10]{FJ} except that now \textit{both} lines carry a $1$.  Then at stage $3$ \textit{only} the line $\ell_{6,8}$ is gated (since $\ell_{7,9}$ has already been gated).  Finally $3,8$ is joined by a line carrying a $1$ and $6,12$ by a line carrying an $*$, similar to the first part of the third from bottom paragraph of \cite [4.4.10]{FJ}.

However in the maximal gating, we do gate $\ell_{6,8}$ and then in the second stage $\ell_{8,11}$ is replaced by  $\ell_{7,11}$ and $\ell_{8,9}$, as above.  However in contrast \textit{both of these lines are made to carry a $1$}.

 Both our constructions lead to a Weierstrass section $e+V$, but the second is more natural and for example behaves well on removal of the last column (Lemma \ref {5.4.8}) leading eventually to a more explicit description of the Weierstrass section, though details will be postponed to a subsequent paper.

 Unlike the minimal gating, in the maximal gating at stage $s+1$, all lines carrying a $\ast$ in $R^{s-1}$ have already been gated.

 \subsection {Two Key Conditions on Composite Lines}\label {5.2}

 Following \cite [4.2]{FJ} we say that two boxes in the tableau $\mathscr T_\mathfrak m$ are adjacent if they lie on the same row with no boxes strictly between them.  The horizontal line joining adjacent boxes defines a certain root vector in $\mathfrak m$. Eventually we consider not necessarily horizontal lines, not necessarily adjoining adjacent boxes.  However only boxes in different columns will be joined. These lines still correspond to further root vectors in $\mathfrak m$.  We may consider that such a line has an arrow pointing from left to right.

 Following \cite [4.2.3]{FJ} we define a composite line to be a concatenation of the lines described above with arrows on successive lines all pointing from left to right.  A composite line defines an element of $S(\mathfrak m)$ being the product of the elements defined by the lines.  Two composite lines are said to be disjoint if they pass through a different set of boxes.


 The aim of the present section is to ensure that the following two properties introduced in  \cite [4.2.3]{FJ} are satisfied.

 \

 $(P_1)$. For all $s \in \mathbb N^+$, all the boxes  in $R^s$ between any two columns $C,C'$ of height $s$ lie on a unique disjoint union of composite lines.

 \

   For $\ell_{i,j}$ occurring in the above, let $I$ (resp. $J$) be the set of left (resp. right) hand entries.  Then, apart from diagonal entries, the  product $\prod x_{i,j}:i \in I, j\in J$ defines a unique monomial in the development of the Benlolo-Sanderson minor and in addition makes the only non-zero contribution to this minor.  (A proof with more precise details is given in \cite [Lemma 4.2.5]{FJ}.)

 \

$(P_2)$. Of the lines which constitute the composite lines in the disjoint union joining the boxes in neighbouring columns $C,C'$ of height $s$ given by $(P_1)$, all are labelled by $1$ except for exactly one which is labelled by $*$.

\

This means that the evaluation of the said monomial is a co-ordinate vector in $\mathfrak m$.  The space spanned by these co-ordinate vectors is $V$, by definition of $V$.  (Again more precise details can be found in \cite [Lemma 4.2.5]{FJ}.)


 \subsection {The Steps}\label {5.3}

 As in \cite [4.2,4.4]{FJ} our construction is carried out in three steps.  Only the last step differs from \cite [4.4]{FJ}.

At the first step \cite [4.2.1] {FJ}
all the horizontal lines on adjacent boxes in a given row are joined. The sum $\tilde{e}$ of the co-ordinates defined by these lines is after Ringel et al \cite {BHRR} a Richardson element, that is $\overline {P.\tilde{e}}=\mathfrak m$.  This indicates that we are on the right track.

At the second step \cite [4.2.2]{FJ} the lines are labelled by either $1$ or $*$, using the rightmost labelling defined there.  We recall that this means a horizontal left going line to a column $C$ of height $s$ in row $R_s$ is labelled by $*$ exactly when $C$ has a left neighbour.

Of course in the second step one might use an intermediate labelling in which just one of the lines in the composite line in row $R_s$ joining neighbouring columns of height $s$, is labelled by $\ast$.  Possibly this also gives rise to a Weierstrass section though we did not check that it does or does not. \textit{However a rightmost labelling is crucial to the truth of Proposition \ref {6.9.2} and Corollary \ref {6.9.3} as even the simplest examples show, viz the composition $(1,2,2,1)$.}

  One easily sees that the total number of lines labelled by $*$ is the total number of pairs of neighbouring columns. Let $e$ (resp. $V$) denote the sum of the co-ordinate vectors (resp. subspaces) defined by the $1$'s (resp. $\ast$).  Through the result of Ringel et al noted in the previous paragraph, we obtain (see \cite [Prop. 5.1]{FJ}) that $\overline{P'.(e+V)}= \mathfrak m$. Thus the restriction map is injective but generally fails to be surjective.

Through the first step any two boxes on the same row are joined by a unique composite line which is horizontal.  Moreover all composite lines are horizontal.  Thus property $(P_1)$ holds, but property $(P_2)$ may fail, because joining a pair of neighbouring columns there may be two or more composite lines carrying an $*$.  This occurs for the composition $(2,1,1,2)$ as illustrated by the central figure in \cite [4.2.4]{FJ}. Just having $(P_1)$ means that the restriction map is surjective at the level of fraction fields \cite [Lemma 4.2.5(ii)]{FJ}.

In a similar fashion to \cite [4.4]{FJ} we recover property $(P_2)$ without losing property $(P_1)$ by an inductive procedure on ``stages''.  This is exemplified by the right hand figure in \cite [4.2.4]{FJ}.

 \subsection {The Stages}\label {5.4}

Step $3$ is carried in stages as indicated below.  Stage $0$ is the end of step $2$. Take $i  \in \mathbb N^+$ and consider we are performing the $i^{th}$ stage and recall the notation of \ref {2.2}.

 First the lines lying in $R^{i-1}$ are left unchanged except that some carrying a $*$ may be gated and some lines in $R_i$ (which are all horizontal and join adjacent boxes) are replaced by new lines in $R^i$ carrying either a $1$ or a $*$. These new lines need \textit{not} be horizontal and \textit{need not} join adjacent boxes.

In \cite [Sect. 4]{FJ}, we gated lines in $R^{i-1}$ carrying an $*$ in a minimal fashion (that is just enough to recover $(P_1)$).  Now we gate \textit{all} lines in $R^{i-1}$ carrying an $*$.  This is gating in a maximal fashion.

Apart from using this maximal gating and the labelling of the new lines in $R^i$, the construction, which is canonical, follows essentially the procedure of \cite [4.4]{FJ}, so in principle we could be brief!

Yet the construction is quite complex and has, as we shall see, some astonishing properties. Therefore we write out full details which in any case will be rather different to those in \cite [Sect. 4]{FJ}.

Following the notation of \cite [4.4]{FJ}, given two boxes $b,b'$, not necessarily in the same row, we write $b<b'$ (resp. $b\leq b'$) if $b$ lies in a column strictly to the left (resp. to the left) of $b'$. Note that equality does not mean that the boxes are the same, merely that they are on the same column.  The strictly inequality means that $b$ can be connected to $b'$ by a line which may form part of a composite line.

\

\textbf{N.B.}  Throughout a gated or ungated line will always refer to a line carrying a $*$.

\subsubsection{Extremal Boxes}\label{5.4.1}

Call a box left (resp. right) extremal if it has no left (resp. right) going lines.  The left (resp. right) extremal boxes at step $2$ are exactly those which have no adjacent box to the left (resp. right) in $\mathscr T$ and in the same row.

By contrast with \cite [4.4]{FJ} our  construction may add a left (resp. right) going line to a box which has no left (resp. right) going line.

A single box $b$ isolated in its row is viewed as a line.

\subsubsection{An Intermediate Condition on Composite Lines}\label{5.4.2}


In the $i-1$ stage the lines and their labelling in $R_j:j\geq i$ are unchanged from those given in step $2$.  However those in $R^{i-1}$  are changed and we show that the following modification of $(P_1)$ holds.

\

$(P_1^{i-1})$.  There is a unique disjoint union of $(i-1)$ composite lines on which all the boxes in $R^{i-1}$ lie.

\

We stress that in this the boxes in $R^{i-1}$\textit{ may} lie on other composite lines, but there is only one disjoint union on which they all lie.

\subsubsection{Up-going linkage, Non-Overlapping and Adjacency}\label{5.4.3}

We shall show (Lemma \ref {5.4.8}(iv)) that the lines in $R^{i-1}$ \textit{not} gated (which we also refer to as ungated),
join $u$ pairs of boxes $b_j,b_j':j=1,2,\ldots,u:b_j\in R^{i-1}, b_j' \in R_{i-1}$ with $b_j<b_j'\leq b_{j+1}$.

Here the first property is non-trivial.  It means that the ungated lines link $R^{i-1}$ to $R_{i-1}$.  It is called up-going linkage to $R_{i-1}$ at the $i^{th}$ stage, or simply up-going.

The inequality which is also non-trivial, is called the strict non over-lapping of the ungated lines at the $(i-1)^{th}$ stage, or non-overlapping.

Finally we show for all $j \in [1,u]$, there is no column of height $\geq i$ between $b_j,b_j'$.  It allows us to speak of ungated lines in $R^{i-1}$ being between columns of height $\geq i$ (and not cut by such columns). It is called the adjacency property at the $(i-1)^{th}$ stage, or adjacency, for short.

These properties are all immediate at step $2$ since all lines are horizontal and join adjacent boxes.

\subsubsection {Ungated lines at the $(i-1)^{th}$ stage} \label{5.4.4}

If $u=0$, in \ref {5.4,3}, that is to say there are no ungated lines in $R^{i-1}$, then $R_i$ (which consists of a single composite horizontal line) is adjoined to $R^{i-1}$. Then since $(P_1^{i-1})$ holds, so does $(P_1^i)$.

Otherwise in the sense of \ref {5.4.3}, the lines between the pairs $b_j,b_j':j=1,2,\ldots,u:b_j, b_j' \in R^{i-1}$ carry a $\ast$ and are ungated.

By the strict non over-lapping of ungated lines at the $(i-1)^{th}$ step, as defined in \ref {5.4.3}, we have
 $$b_1<b_1'\leq b_2<b_2'\leq \ldots b_{u-1}<b'_{u-1}\leq b_u <b'_u. \eqno {(1)}$$

\subsubsection{The Labelling of Columns}\label{5.4.5}

We shall label the set $\mathscr C$ columns of height $\geq i$ following the particular manner resulting from \cite [4.4.10, Operation $3$]{FJ} except that we shall present this in just one operation.  In this \textit{some columns are unlabelled (and then ignored)} and some columns can have two different but consecutive labels  (for example as in the second sentence in the paragraph below).  Thus, as we have already forewarned the reader (\ref {2.2}), the labelling is different to that in either \ref {2.2} or that in \ref {2.7}.

Let $C_0$ be the leftmost element of $\mathscr C$.  Let $C_1$ be the rightmost element of $\mathscr C$ (possibly $C_0$ itself) such that between $C_0,C_1$ there are no ungated lines in $R^{i-1}$.  Then either there are no elements in $\mathscr C$ to the right of $C_1$ and we are done or there is a leftmost one $C_2$ such that between $C_1,C_2$ there are at least some ungated lines in $R^{i-1}$.  Moreover by choice of $C_1,C_2$ there are no elements of $\mathscr C$ between $C_1,C_2$.  In particular since $R_i$ is so far unchanged from its presentation in step $2$, there is a single line joining the boxes $C_1\cap R_i, C_2\cap R_i$ which is to be  deleted (see \ref {5.4.6}).

Repeating this process, we obtain columns $C_j:j=0,1,\ldots, 2n+1$ in $\mathscr C$ with the following properties.

\

$A$. If $j$ is even one can have $C_j=C_{j+1}$ or being distinct there are no ungated lines in $R^{i-1}$ between them.   In addition there \textit{can} be elements of $\mathscr C$ between $C_j,C_{j+1}$.

\

$B$.  If $j$ is odd, the columns $C_j,C_{j+1}$ are distinct with no elements in $\mathscr C$ between them; but at least one ungated line in $R^{i-1}$ between them.

\

In addition $C_0$ (resp. $C_{2n+1}$) is the leftmost (resp. rightmost) element of $\mathscr C$, with possibly $C_{2n}=C_{2n+1}$.

In the above the reader might view \cite [Figure 1]{FJ}. In this case $C_6=C_7$, but this is not indicated in that figure.

\subsubsection{Deletion of Lines}\label{5.4.6}

Now recall the boxes in $b_j, b_j' \in R^{i-1}:j\in [1,u]$ appearing in equation $1$ of \ref {5.4.4}.  Recall that $b_j,b_j': j \in [1,u]$ is joined by a line carrying a $\ast$ and (for the moment) is ungated.

Label the columns $C_j \in \mathscr C:j \in [1,2n+1]$ as above. Let $b''_j$ denote the unique box in $C_j\cap R_i$.

Gate the lines between the pairs $b_j, b_j' \in R^{i-1}:j=1,2,\ldots,u$ carrying a $\ast$.  The meaning of this gating is that we may
regard $b_j$ (resp. $b_j'$) as having no right (resp. left) going line.  This will lead to a redirection of the composite lines as we go to the next stage.

Now \textit{delete} the lines in $R_i$ joining $b''_j, b''_{j+1}$ for $j$ \textit{odd}.
Then for $j$ odd (resp. even), $b''_j$  has no right (resp. left) going line.

Notice that this does \textit{not} delete all the lines in $R_i$.  In particular this is true if $j$ is even and $b_j''\neq b''_{j+1}$ and also because some columns which meet $R_i$ are unlabelled.  Some of these lines may carry a $\ast$ by virtue of the second step (\ref {5.3}).

Note one can have $b''_{2j}=b''_{2j+1}:j \in [0,n]$.  Yet this is compatible with not deleting the line between them.

Notice in addition that $b''_0$ (resp. $b''_{2n+1}$) has no left (resp. right) going line.

 \subsubsection{Rejoining Boxes}\label{5.4.7}

 We now join boxes in $R^i$ by lines.  This joining is rather natural and results from boxes not have left or right going lines except those which are gated.  The rule can viewed as being specified by a ``minimal distance criterion'' described at the end of this section and having said that the reader may already guess this rule.

 \

 If there exists $j \in [1,u]$ such that $b_j<b_0''$, let  $u_0\in [1,u]$ be maximal such that $b_{u_0}<b_0''$ and join $b_{u_0}<b_0''$ by a line.

If there exists $j \in [1,u]$ such that $b'_j>b_{2n+1}''$, let  $u_{2n+1}\in [1,u]$ be minimal such that $b'_{u_{2n+1}}>b_{2n+1}''$ and join $b'_{u_{2n+1}}>b_{2n+1}''$ by a line.

This is complimented in the natural way  by:


  For $j\in [2,2n]$ even, let $u_j\in [1,u]$ be maximal such that $b_{u_j}<b''_j$.

  For $j\in [1,2n-1]$ odd, let $u'_j\in [1,u]$ be minimal such that $b''_j<b'_{u'_j}$.

  Since there are no ungated lines between $C_{j},C_{j+1}$, for $j$ even, one has $u'_{j+1}=u_j+1$.  Extending the first two rules in the beginning of this paragraph, we impose:

  \

 (i).  Join $(b_{u_j},b''_j)$, for $j$ even  and join $(b''_j,b'_{u'_j})$, for $j$ odd, by a line.


 \

 Finally ``to gather loose ends'' we perform \textit{after} (i):

 \

  (ii). Join $b_j,b'_{j+1}$ by a line if $j \in [1,u]\setminus \{u_j\}_{j=1}^{2n}$.

  \

  The above rules satisfy the following

  \

  \textbf{Minimal Distance Criterion}.  The pair $b,b'$ of  boxes to be joined in (i) by a right going line from $b$ to $b'$ is given in terms of the distance between the columns in which they are contained. Thus starting from  $b'$ having no left going line which is not gated, then $b$ is chosen to be of minimal distance from $b'$ with no right going line which is not gated, or  starting from  $b$ having no right going line with is not gated, then $b'$ is chosen to be of minimal distance from $b$ with no left going line which is not gated.  This is essentially true for (ii).  Thus $b_j$ is of minimal distance to $b'_{j'}$ for $j<j'$ exactly when $j'=j+1$, whilst $b_j,b'_j$ are already joined by a line, though carrying a $\ast$, so we desist joining them by a line carrying a $1$.

   \subsubsection{Labelling of the New Lines}\label{5.4.8}

   To recall: In the previous two sections we described in the $i^{th}$ stage exactly which lines may be gated or deleted and which new lines are created.  Recall (\ref {5.4.6}) that the deleted lines are all horizontal.

   We now describe the labels that we put on the new lines. We describe only those which are to carry a $\ast$, the remaining new lines are to carry a $1$.

   We emphasize that this will not exhaust all the lines meeting $R_i$ which carry an $*$, since they may belong to the subset which are not deleted.  For example suppose $C_{2n}\neq C_{2n+1}$.  If $C_{2n+1}$ has height $i$ and has a left neighbour in $\mathscr C$, then $b_{2n+1}$ is joined to $b_{2n}$ by a line which is not deleted and carries a $*$.

   \

   Suppose the deleted line $\ell_{b''_j,b''_{j+1}}$:$j$ odd, carries a $\ast$ (resp. $1$), then the line which replaces it, namely $\ell_{b_{u_{j+1}},b''_{j+1}}$ is assigned to carry a  $\ast$ (resp. $1$).    This assignment is a little arbitrary but \textit{essential} to what follows. (In particular for (ii) below and for Lemma \ref {6.5.2}.)  We juggled with alternative assignments and the reader is invited to do the same.) From it we may immediately draw eight conclusions.

    We say that a line is at level $j$ if its right end-point is in $R_j$.

   \begin {lemma}

   \

   $(i)$.  The number of lines in $R^j$ carrying a $\ast$ remains unchanged, as one goes through the stages, and so equals the number of neighbouring columns of height $\leq j$.

   \

   $(ii)$.  Two lines carrying a $\ast$ cannot share the same right hand end box.

  \

  $(iii)$.  A line $\ell$ carries a $\ast$ if and only if its right hand end is a box at the lowest point of a column $C$, that is in $R_i\cap C$, where $i= ht C$, and $C$ has a left neighbour of height $i$.  Viewed as going from left to right it also goes upwards.

  \

  $(iv)$.   The new ungated lines for the $(i+1)^{th}$ stage satisfy the up-going linkage property and the strict non-overlapping property.

  \

  $(v)$.  Between any two neighbouring columns $C,C'$ of height $i$, every box in $R^i$ has a right (resp. left) going line carrying a $1$ with the possible exception of those in $C'$ (resp. $C$) and the ungated line $\ell$ carrying a $\ast$ (noted in (iii)).

  \

  $(vi)$.  No two lines carrying a $1$ may either start nor end at the same box.

  \

  $(vii)$. The adjacency property of ungated lines holds at the $i^{th}$ stage.

  \

  $(viii)$.  Consider a box $b \in R^i$ with a right (resp. left) going line carrying a $1$.  Then there is no line carrying a $\ast$ from a row $R_{j}:j>i$ with left (resp. right) end point $b$.

   \end {lemma}

  \begin {proof}  The last part of (iv) follows by the minimal distance criterion of \ref {5.4.7}. Assertion (v) is immediate from (i),(ii) of \ref {5.4.7}.  For (vi), note that by our construction a left (resp. right) going line carrying a $1$ cannot be drawn from a given box already possessing a left (resp. right) going line which does not carry a $\ast$.  For (vii) we note that by \ref {5.4.8}, the only new lines carrying a $\ast$ are obtained by the construction of \ref {5.4.7}(i).  Then the assertion is immediate from the way in which columns of height $\geq i$ are labelled and the adjacency property of ungated lines at the $(i-1)^{th}$ stage.  $(viii)$ obtains because a line carrying a $1$ is not gated.

  \end {proof}

  \textbf{Remarks.}  By contrast to (ii) several lines carrying a $\ast$ may share the same left hand end box (see \cite [Fig. $3$]{FJ}).

  \


   (ix). In contrast to (iv) a line carrying a $1$ can be right and down going; \textit{but} only by one row and by exactly one row only when it shares a right hand -end point with a line carrying a $\ast$.

   \

   These are the lines $\ell_{b''_j,b'_{u'_j}}$, with $j$ odd,  given by \ref {5.4.7}(i), noting that $b''_j \in R_i$, by definition and $b'_{u'_j}\in R_{i-1}$, by the up-going linkage property.

   \

 (x).  A line carrying a $1$ can be right and up-going; \textit{and} by several rows.

 \

  These are lines $\ell_{b_{u_j},b''_j}$, with $j$ even,  given by \ref {5.4.7}(i), noting that $b''_j \in R_i$, by definition and $b_{u_j}\in R^{i-1}$ by the up-going linkage property. For the composition $(1,2,1,1,2,3)$ the line $\ell_{3,9}$ is right and up-going by two rows.

  It is easy to check that if the minimal and maximal gatings coincide then, our present construction gives exactly the result described in \cite [4.4]{FJ}.

\subsubsection{Removal of the Last Column}\label{5.4.9}

Let $\hat{\mathscr C}=(C_1,C_2,\ldots,C_k)$ denote the complete set of columns in $\mathscr D$.

Let $\ell(\hat{\mathscr C})$ denote the set of labelled lines in $\mathscr T_\mathfrak m$ given through \ref {5.4.1} - \ref {5.4.8}.

Let $\hat{ S}$ denote the set of columns obtained from $\hat{\mathscr C}$ by deleting the last column $C_k$.

\begin {lemma}  $\ell(\hat{ S})$ is obtained from $\ell(\hat{\mathscr C})$ by deleting $C_k$ and all the lines from $\hat {S}$ to $C_k$.
\end {lemma}

\begin {proof}  The proof is by induction on the stages.  Fix $i \in \mathbb N^+$.  Recall \ref {5.4.5} and let $\mathscr C$ (resp. $S$) denote the subset of columns of height $\geq i$ in $\hat {\mathscr C}$ (resp. $\hat{S}$).

If $i>c_k$, then $C_k \notin \mathscr C$.  Then $\mathscr C, S$ coincide.  Then via the construction in \ref {5.4.5}-\ref {5.4.7}, the lines deleted in $R_i$ at the $(i+1)^{th}$ stage, is the same for both. In this   consider the new line from  $b''_j$ joining it to an element of $R^{i-1}$ . If $j$ is odd this line goes to the right and meets the box $b'_{u'_j}$. Thus if $b'_{u'_j}$ is \textit{not} in the last column, it is unchanged by the minimal distance criterion of \ref {5.4.7} determining $u'_j$. If $j$ is even, this line goes to the left meeting $b_{u_j}$ and this is obviously unchanged by the presence of the last column.

It remains to consider the case $i \leq c_k$.  Then $C_k$ is the rightmost element of $\mathscr C$ which to compare with \ref {5.4.5} must be relabelled as $C_{2n+1}$.

Suppose $C_{2n}\neq C_{2n+1}$. Then the line joining $b''_{2n},b''_{2n+1}$ is not deleted.  Then the proof proceeds as above except that $b'_{u'_j}$ cannot belong to the last column because there are no gated lines between $C_{2n},C_{2n+1}$.   Finally suppose  $C_{2n}= C_{2n+1}$. Now deleting the line joining $b''_{2n-1},b''_{2n}$ as prescribed by \ref {5.4.6} means that  $b''_{2n-1}$ will have no right going line.  However this is also true when the last column is not present.

Finally $b''_{2n-1}$ is joined to $b'_{u'_j} \in R^{i-1}$, as prescribed by \ref {5.4.7}(i). If $b'_{u'_j}$ does not belong to the last column, then by the minimal distance criterion of \ref {5.4.7} determining $u'_j$, it is unchanged by removal of the last column.

Hence the assertion of the lemma.
\end {proof}

\subsubsection{}\label{5.4.10}

It turns out that $\ell(\hat{\mathscr C})$ does not behaves quite so well under removal (or adjunction) of columns on the left.

To make the comparison with \ref {5.4} less confusing we adjoin a column $C_0$ on the left  of $\hat{\mathscr C}$ and denote the resulting set of columns by  $\hat{T}$.


To have an idea of what we should obtain we first consider a simple case.

\

\textbf{Example.}  Consider the composition $(2,3,1,1,3,2)$ and remove the first column.  After removal of the first column, the line $\ell_{6,9}$ carried a $\ast$ now carries a $1$.  However there is a more dramatic change.  Indeed the line $\ell_{5,9}$ that carries a $1$ and the line $\ell_{6,12}$ that carries a $\ast$ are replaced by a line $\ell_{5,12}$ carrying a $\ast$.  However if we just concentrate on the lines contained entirely in $R^2$, this last more dramatic change is not seen.

\

Recall \ref {5.4.8}(ii) that to any box $b$ there is at most one left going line carrying a $\ast$. It exists exactly when the column $C$ in which it is contained has a left neighbour and $b \in C\cap R_i$, where ht $C=i$.

\begin {lemma}  Suppose $C_0$ has height $i$. Concerning only the lines lying entirely in $R^i$, one obtains $\ell(\hat{\mathscr C})$ from $\ell(\hat{T})$ by deleting $C_0$ and all the lines from $\hat {\mathscr C}$ to $C_0$ with one exception, when $\hat{\mathscr C}$ admits a (leftmost) column $C_t$ of height $i$.  Then the unique left going line from $R_i\cap C_t$ carrying a $\ast$ in $\ell(\hat{T})$, carries a $1$ in $\ell(\hat{\mathscr C})$.
\end {lemma}

\begin {proof}  The proof follows closely that of the last part of the proof of lemma \ref {5.4.9}. Thus it is by induction on the stages used in the construction of Section \ref {5.4}.  Fix $j \in \mathbb N^+$.  Recall \ref {5.4.5} and let $T$ (resp. $\mathscr C$) denote the subset of columns of height $\geq j$ in $\hat {T}$ (resp. $\hat{\mathscr C}$).

Suppose $j\leq c_0:=i$. At the $j^{th}$ stage the boxes in $R_k:k>i$ and the lines joining them, which are all horizontal, play no role in the construction, so can be ignored.

As in \ref {5.4.5} label the elements of $T$ as $C_0,C_1,\ldots$
Then  $C_0$ is the leftmost element of $T$, and
  $C_1$ (of \ref {5.4.5}) is either $C_0$ or is distinct from $C_0$ and there are no ungated lines in $R^{j-1}$ between them. Notice that the gating at the $j<i$ stage is the same with or without $C_0$ because the labelling of the lines lying entirely in $R^{i-1}$ is unchanged.

In the first case the algorithm of \ref {5.4.5}-\ref {5.4.7} starts at $C_1$ and so $C_0$ can be ignored.  In the second case the line joining $b_0'',b_1''$ is not deleted and so again $C_0$ can be ignored.   The only difference comes in applying \ref {5.4.8} concerning the labels on the subsequent lines and the only difference is the one cited in the conclusion of the lemma.

\end {proof}

\textbf {Remark.} One may add that the above possible difference   at the $(i+1)^{th}$ stage is that the gating in  $\hat {T}$ compared to $\hat{\mathscr C}$ can be changed, and so the subsequent labelling of columns following \ref {5.4.5}, can be  changed.  In the example the gating between $6,9$ is lost at stage $3$ when the first column is deleted.  This accounts for the subsequent more dramatic difference, noted in the example.

  \subsubsection{Proof of $P_1^i$}\label{5.4.11}

\begin {prop} After the $i^{th}$ stage has been carried out $(P_1^i)$ holds.

Moreover the up-going linkage to $R_i$, the strict non-overlapping property of the ungated lines and adjacency holds at the $(i+1)^{th}$ stage.
\end {prop}

\begin {proof}  The proof follows closely that of \cite [4.4.11]{FJ}.

It is immediate from the two previous lemmas that we can assume that the first and the last column of $\mathscr D$ have height $i$. Then all the boxes in $R^i$ lie between neighbouring columns of height $i$.

Between any two neighbouring columns of height $i$, the existence of a disjoint union of $i$ composite lines joining the two columns is immediate from  Lemma \ref {5.4.8}(v).  Moreover if we forbid passage across the gated lines in $R^{i-1}$, uniqueness follows from Lemma \ref {5.4.8} (vi).

Then we may concatenate these uniquely determined composite lines to form a disjoint union of $i$ composite lines joining boxes in the first and last columns.  Label these composite lines by the entry in $[1,i]$ of the rightmost column from which they end.
Notice this means in particular that, in the composite lines, we are not allowing lines which jump over a column of height $i$, though they may exist - viz $\ell_{6,17}$ of the example below.  (Otherwise uniqueness of the disjoint composite lines from the first to the last column \textit{fails}, and so does the fact that they go through all the boxes in $R^i$.)

Let us show, by downward induction, that the removal of the (horizontal) lines  $R_k:k \leq i$ in the $k^{th}$ stage and their replacements to join boxes on the first and last columns by the $k^{th}$ composite line means that the gated lines of the $(k-1)^{th}$ stage cannot be used in forming the required disjoint union.  Moreover in this, the $k^{th}$ composite line is uniquely determined by the lines remaining after the $\ell^{th}:i\geq \ell >k$ composite lines have been drawn.  In particular the $i^{th}$ composite line is uniquely determined even without forbidding passage across the lines gated at the $i^{th}$ stage. All this should be obvious on gentle reflection.  It might also be helpful to glance at \cite [Figure 1]{FJ}.

Here we can assume only the first $C$ and the last column $C'$ have height $i$, by the last sentence of paragraph two.

Consider the imposition of the $k^{th}$ stage and in the sense of \ref {5.4.5}, let $\mathscr C$ denote the columns of height $\geq k$ labelled as there.  Recall $A,B$ of \ref {5.4.4} that the (to be) gated lines (cf \ref {5.4.6}) at the $(k-1)^{th}$ stage lie between the columns $C_j,C_{j+1}$ for $j\in [1,2n-1]$ odd.  Then $\ell_{b''_j,b''_{j+1}}$ is removed (\ref {5.4.6}), so  $b''_j$ (resp. $b''_{j+1}$) has no right (resp. left) going line.  Thus in order that they lie on a union of composite lines joining $C$ and $C'$ we are forced (as carried out in \ref {5.4.7}(i)) to join them to $b'_{u'_{j+1}}$ and $b_{u_j}$ respectively.  Recalling that $u'_{j+2}=u_{j+1}+1$, for $j$ odd,  this means that the gated lines meeting $b'_{u_{j+1}+1},b_{u_j}$ (on both sides) cannot be used in a union of composite lines joining $C$ and $C'$.  Then in order for the latter to lie on this union one is forced to successively join loose ends arising between $C_j,C_{j+1}$.  Consequently none of the gated lines  at the $(k-1)^{th}$ stage, lying between the columns $C_j,C_{j+1}$, can be used in a union of composite lines joining $C$ and $C'$.

The last part of the proposition was already established in Lemma \ref {5.4.8}(iv,vii).
\end {proof}

\textbf{Example 1.} Consider the composition $(1,2,1,1,1,1,1,2)$ and see Figure $6$.  The horizontal lines in step $2$ all carry a $\ast$ except $\ell_{1,2},\ell_{8,9}$.  In stage $3$, $\ell_{3,10}$ is deleted and the remaining lines carrying a $\ast$ (in $R_1$) are gated. Then \ref {5.4.7}(i) prescribes joining $(3,4)$ and $(7,10)$ the latter with a $\ast$ by \ref {5.4.8}.  Then the gated lines $\ell_{2,4},\ell_{4,5},\ell_{6,7},\ell_{7,8}$ cannot be used in joining the two columns of height $2$.  What about the gated line $\ell_{5,6}$?  Since $4$ has no right going line, \ref {5.4.7}(ii) prescribes joining $(4,6)$.  This finally excludes the use of  the gated line $\ell_{5,6}$.  We may remark that the composite lines joining the columns of height $2$ can be defined by the sequence of boxes they pass through, Namely $(2,5,7,10)$ and $(3,4,6,8,9)$. Notice that the composite lines exchange the entries of the last column compared to step $2$.

\

\textbf{Example 2.}  Consider the composition $(3,2,1,1,2,3,2,3)$ and see Figure $7$.  The third composite line passes successively through the boxes $(3,9,11,17)$, passes through the ungated line $\ell_{11,17}$ and excludes the subsequent use of the gated lines $\ell_{6,9}$ and $\ell_{11,14}$.  The second composite line passes successively through the boxes $(2,5,7,8,10,13,15)$ and excludes the subsequent use of the gated line $\ell_{6,7}$. The first composite line passes successively through the boxes $(1,4,6,12,14,16)$ and contains the gated line $\ell_{6,12}$.

Notice that for the first composite line we are not allowed to use the line $\ell_{6,17}$ because it jumps over a column of height $3$ and so cannot be a concatenation of composite lines joining boxes in neighbouring columns.  It must of course also miss a box in the column it jumps over.

\

We have now completed our construction and the main result of this section which we should like to emphasize covers a dazzling variety of situations, which would have been dramatically fewer (and indeed rather boring) had we been able to just use partitions instead of compositions.

\

The reader can now breathe a sigh of relief.

  \subsubsection{Existence of a Weierstrass Section}\label{5.4.12}

For a pair of neighbouring columns $C,C'$ of height $s$, we obtain $(P_1)$ from $(P_1^s)$.  The proof, which is fairly easy, follows exactly that of \cite [Cor. 4.4.12]{FJ}, Lemma \ref {5.4.8}(vi),(vii),(viii).  The key point is that these conditions exclude a composite line joining boxes in $C,C'$ passing through some $R_{s'}:s'>s$.  This is an important role of adjacency.

 Moreover $(P_2)$ obtains.  Indeed $C'\cap R_i$ admits a left going line carrying a $\ast$ and this is exactly the one line carrying a $*$ in these disjoint composite lines prescribed by $(P_1)$.

This completes (our modified) step $3$.

The proof that $e+V$ obtained from step $3$ is a Weierstrass section,  is the same as in \cite [Thm. 6.2]{FJ}.  Indeed if one admits \ref {1.2} (Fact) due to Melnikov \cite {M1} that the number of invariant generators is the number of pairs of neighbouring columns.  Indeed it then becomes obvious through $(P_1),(P_2)$ using \cite [Lemma 4.2.5]{FJ}.

A second proof of \ref {1.2} (Fact) results from \cite [Prop. 4.3.5]{FJ} as used in \cite [Prop. 5.1]{FJ}.


\section {Regularity}\label {6}

In the following it is vital to distinguish rows $R_i$ and columns $C_j$ of $\mathscr T$ with the rows $\textbf{r}_u$ and columns $\textbf{c}_v$ in the full matrix $\textbf{M}$ defining $\mathfrak {gl}(n)$.  Here the right going line $\ell_{u,v}$ from $R_i\cap C_j$ with entry $u$, to $R_s\cap C_t$ with entry $v$, corresponds to the co-ordinate function $x_{u,v}$ lying in  $\textbf{r}_u \cap \textbf{c}_v$.


Recall that the entries in $\mathfrak m$ corresponding to a line in $\mathscr T_\mathfrak m$ carrying a $1$ (resp. a $\ast$) define e (resp. $V$).

When a lower left hand corner minor of the appropriated size is evaluated on $e+V + \textbf{1}$, its leading homogeneous term is the Benlolo-Sanderson invariant.

Here $1$ is placed at the co-ordinates defined by $e$ \textit{and} on the diagonal. At a place carrying a $\ast$ the co-ordinate is left unchanged. Finally elsewhere it is set equal to $0$.

For the present, we take the lines joining boxes in $\mathscr T_\mathfrak m$ and their labels to be given by the construction of \ref {5.4}.  Eventually we may augment $e$ by replacing it by $e_{VS}$ - see \ref {1.7}.

Recall $\mathfrak u_{\pi'}$  defined in \ref {2.8}.

We show that $e \in \mathfrak u_{\pi'}$ as well as $e \in \mathscr N$.  In a subsequent paper we show that $e_{VS}\in \mathfrak u_{\pi'}$ and $e_{VS} \in \mathscr N$.  Here we just verify the assertions of this last sentence in the special cases needed.

\

\textbf{Example.}  Let $\mathfrak p$ be defined by the composition $(3,2,1,1,2,3)$ which is actually a relatively easy example.  In this case the construction of \ref {5.4} gives
$$e =x_{1,4}+x_{2,5}+x_{3,9}+x_{4,6}+x_{5,7}+x_{7,8}+x_{8,10}+x_{9,11},
$$
with $V$ being the linear span of $\{x_{6,7},x_{6,9},x_{6,12}\}$.

It turns out that $e$ is not regular (see Example 1 of \ref {6.10.9}).  We must add to it the ``VS element'', $x_{6,11}$ to make it so. The resulting root vectors all lie in $\mathfrak u_{\pi'}$.

In this case $P.e_{VS}$ is dense in an orbital variety closure but this fails in general. Again the roots belonging to vectors occurring in $e_{VS}$ are linearly independent, but this may fail in general.

 Finally none of the co-ordinates of $V$ lie in $\mathfrak u_{\pi'}$.  In Corollary \ref {6.9.6} this is shown to hold in general using Proposition \ref {2.6}.

\subsection {Inclusion of $e$ in the Nilfibre}\label {6.1}

Take $e \in \mathfrak m$ to be given by the construction of \ref {5.4}.

\begin {lemma}  $e \in \mathscr N$.
\end {lemma}

\begin {proof}  This generalizes \cite [Lemma 5.4.2]{FJ}. The argument is similar and a rather easy consequence of $(P_2)$.

Fix a pair of neighbouring columns $C^{1},C^{2}$ of height $s$, with $C^{1}$ to the left of $C^{2}$. For all $j=1,2,\ldots,s$, let $b^1_j$ (resp. $b^2_j$) denote the box in $C^1$ (resp. $C^2$) in row $R_j$.

By Lemma \ref {5.4.8}(vi), there is at most one composite line formed from lines carrying only $1's$ from $b^1_j$ to a box in $C^2$.  If this latter box exists we denote it by $b^2_{\varphi(j)}$.

Again by Lemma \ref {5.4.8}(vi), it follows that $\varphi$ is an injection of $[1,s]$ to itself.

Once more by Lemma \ref {5.4.8}(vi), the composite lines  carrying only $1's$ joining $b^1_j,b^2_{\varphi(j)}$ are disjoint.

Recall that a box $b$ having entry $u$ joined to it on the right by a box $b'$ with entry $v$, defines the co-ordinate function $x_{u,v}$.

 Thus a composite line joining $b^1_j$ and $b^2_{\varphi(j)}$ defines a product of entries of a minor in $\mathfrak m$.

  In particular, the Benlolo-Sanderson minor $\gr M_{C,C'}(\mathfrak m)$  defined by the pair of neighbouring columns $C,C'$ is non-vanishing on $e$ if and only if $\varphi$ is surjective.

Now in fact $\varphi$ is almost surjective.  Indeed after step $3$ of \cite [4.2.4]{FJ}, there is a unique disjoint union of composite lines which joining $b^1_j \in C^1$ to those in $b^2_{\hat{\varphi}(j)} \in C^2$, for all $j \in [1,s]$  with $\hat {\varphi}$ a permutation of this set. Exactly one carries a $\ast$.  Suppose the composite line carrying a $\ast$ starts at $t \in [1,s]$.
Then $\varphi(j)=\hat{\varphi}(j)$, for all $j \in [1,s]\setminus \{t\}$ and $\im \varphi = [1,s]\setminus \{\hat{\varphi}(t)\}$.

In particular $\varphi$ is not surjective.  Thus $e$ vanishes on $\gr M_{C,C'}(\mathfrak m)$.  Since these minors generate $\mathbb C[\mathfrak m]^{P'}_+$, the assertion of the lemma follows.

\end {proof}


\subsection{The Set of Roots in $e$}\label{6.2}


Recall that the co-ordinate function $x_{i,j}$ defined in \ref {2.1}, is a root vector $x_{\alpha_{i,j}}$ with $\alpha_{i,j}=\alpha_i+\cdots + \alpha_{j-1}$.

As in \ref {2.2}, let $I$ denote the set of co-ordinates defined by the lines carrying a $1$, that is to say $(i,j) \in I$ whenever there is a line carrying a $1$ joining a box with label $i$ to a box on its right with label $j$.

We also use $I$ to denote the set $\{\alpha_{i,j}\}_{(i,j)\in I}$.

Set $E=\sum_{\alpha\in I} \mathbb C x_\alpha$.

Recall that $H$ is the closed connected subgroup of $P$ with Lie algebra $\mathfrak h$.

\begin {lemma} The roots in $I$ are linearly independent.  In particular $E=\overline {H.e}$ and so is contained in $\mathscr N$.
\end {lemma}

\begin {proof}  The first co-ordinates of the $x_{i,j}:i \in I$ are pairwise distinct by Lemma \ref {5.4.8}(vi) (and so are the last co-ordinates). Thus  no two $1$'s lie on the same row (nor the same column). Hence the assertions.
\end {proof}


\subsection{Strong Linear Independence}\label{6.3}

Actually one can do a little better than the above lemma.

\begin {lemma}  If $\alpha, \beta \in I$, then $\alpha-\beta$ is not a non-zero root. In particular there exists $w \in W$ such that $wI \subset \pi$.
\end {lemma}

\begin {proof} By the observations in the proof of the previous lemma,  $\alpha_{i,j}-\alpha_{k,\ell}$ is not a root. The last part follows from the fact that $\pi$ is simple of type $A$.
\end {proof}

\textbf{Remark.}  Of course $e \in \mathfrak m$ and so is ad-nilpotent.  Following the composite lines carrying a $1$, it becomes easy to compute the  nilpotency class of $e$ and hence $\dim G.e$, via the right hand side of \ref {2.3}, the columns heights being given by the nilpotency class.

\subsection{Cycles in the Graph of the Extended Set of Roots}\label{6.4}

After adjoining the VS terms, the conclusion of Lemma \ref {5.4.8}(vi) fails.    For linear independence (or lack of it) of the roots of the root vectors in $e_{VS}$ we use the following certainly well-known lemma, which we state and prove for completeness.

Let $R$ be a subset of non-zero roots in type $A$.  We define a graph $\mathscr G(R)$ whose edges are the elements of $R$ joined whenever two roots have a non-zero scalar product.

\begin {lemma}  The elements of $R$ are linearly independent if and only if $\mathscr G(R)$ has no cycles.
\end {lemma}

\begin {proof}  The edges of a cycle as roots add to zero, hence necessity.

Sufficiency is proved by induction on the number of  vertices. Moreover the root system can be supposed indecomposable and hence of type $A_{n-1}$, for some integer $n>1$.  Label the roots as in Bourbaki \cite {Bo}, that is as  $\{\varepsilon_i-\varepsilon_j\}$ with $i,j \neq$ in $[1,n]$. Then $\varepsilon_i-\varepsilon_j, \varepsilon_k-\varepsilon_\ell$ have a negative (resp. positive) scalar product if and only if $j=k$ or $i=\ell$ (resp. $i=k$ or $j=\ell$).  Then the vertices of $\mathscr G(R)$ can be labelled by elements of $[1,n]$ with every label appearing only once and quite possibly not at all.

Now suppose that $\mathscr G(R)$ has no cycles.  Then being a finite graph, it must have at least one extremal vertex, that is to say having only one edge joined to it. Remove this edge. This also removes the vertex and denote its label by $j\in [1,n]$. Then we obtain a graph with vertices indexed by a subset of $[1,n]\setminus \{j\}$. This is again a finite graph with no cycles but one less vertex.  Then induction establishes the result.
\end {proof}

\textbf{Remark.} Although the roots of the root vectors occurring in $e_{VS}$ may be not linearly independent, the root vectors themselves are.  From the graph it is not particularly difficult to compute the nilpotency class of $e_{VS}$.  This will be carried out in a few examples.

\subsection {Deletion of the Last Column and Other Things}\label{6.5}

\subsubsection{Removal of the last column}\label{6.5.1}

Recall the notation of \ref {5.4.9}.  Here $\hat{\mathscr C}$ denotes the set of all columns in $\mathscr D$ and $\ell(\hat{\mathscr C})$ the set of labelled lines on $\mathscr T_\mathfrak m$ given through \ref {5.4}. Recall that the latter behaves well on deletion of the last column Lemma \ref {5.4.9}.  As a consequence $\ell(\hat{\mathscr C})$ may be reconstructed by induction on the number of columns.  This will be pursued in a subsequent paper.
Here we shall just obtain what is needed to describe the irreducible component of $\mathscr N$ containing $e$ as a $B$ saturation set, similar to the description of an orbital variety.  Sometimes this is an orbital variety, sometimes not - see \ref {6.10.9}, Examples.

\subsubsection{Removal of the first column}\label{6.5.2}

$\ell(\hat{\mathscr C})$ behaves less well on deletion of the first column (Lemma \ref {5.4.10}.  Yet we have already seen that these two properties can be useful and indeed will also be used in a subsequent paper to prove the remarkable description of potentially bad VS pairs in Proposition \ref {6.10.4}, formulated by the first author.

\subsubsection{Sandwiching}\label{6.5.3}

\begin {lemma}  The new lines constructed in  \ref {5.4.7} (i) (resp. (ii)) at the $i^{th}$ stage do not cross a column of height $\geq i$ (resp $\geq i+1$).

\end {lemma}

\begin {proof}  The assertion for (i) is immediate from the way that columns of height $\geq i$ are labelled.  This applies in particular to line carrying a $\ast$.  This was used to prove \ref {5.4.8}(vii) which we called adjacency.

  The assertion for (ii)  corresponds to the joining of the loose ends $b_j,b'_{j+1}$, say in columns $C_0,C'_1$.  Here there is a  gated line $\ell'_0$ (resp. $\ell'_1$) joining $b_j,b'_j$ (resp. $b_{j+1},b'_{j+1}$).  We  want to show that the new line $\ell_{b_j,b'_{j+1}}$ prescribed by \ref {5.4.7}(ii) does not cross a column of height $\geq i+1$.

  By \ref {5.4.8}(iv), $b_j',b'_{j+1}$ lie in columns $C_0',C_1'$ of height $i$.  By (i) above the columns strictly between $C_0,C_0'$ and between $C_1,C'_1$ have height $<i$.

  Recall (\ref {5.4.4}, Eq. $1$) that we have $b_{j}' <b_{j+1}$. It remains to show that there is no column of height $>i$ between $C_0'$ and $C_1$. Otherwise take the rightmost one $C''$. Then  by the minimal distance criterion of \ref {5.4.7}, the box in $C''\cap R_{i+1}$ would be joined to $b_{j+1}'$ with a line carrying a $1$, contradicting that $b_{j+1}'$ is a ``loose end'', that it to say having no left going line carrying a $1$.

\end {proof}

  \

  \textbf{Example}.  Consider the composition $(4,3,2,1,1,2,3,2,1,1,2,3,4)$ and the lines adjoined by (i),(ii) at stage three.  The line $\ell_{10,16}$ given by  (i) does not cross a column of height $\geq 2$ .  Whereas the line $\ell_{10,25}$ given by (ii) joins  the loose ends emanating from the boxes carrying $10,25$ \textit{does} cross a column of height $3$ but does not cross a column of height $\geq 4$.  A further example results from Figure $7$.

\subsubsection{An Exception}\label{6.5.4}  In the framework of the construction of \ref {5.4} we have the

\begin {cor}  Let $b$ be a lowest box in some column $C$ of $\mathscr T_\mathfrak m$ having  height $t$, that is $b \in R_t\cap C$.  If $b$ admits a left (resp. right) going line $\ell$ carrying a $1$, then every other box in $C$ admits a left (resp. right) going line $\ell'$ carrying a $1$, with the possible (important) exception when $\ell'$ is a right going line carrying a $\ast$ meeting $R_{t-1}\cap C'$ for some column $C'$ of height $t-1$ to the right of $C$.  In this case,  $C'$ admits a left neighbour $C''$ and then $\ell,\ell'$ both have $R_{t-1}\cap C'$ as a right end-point.
\end {cor}

\begin {proof}

\

\textbf{Claim}.  Under the hypothesis, $C$ admits a column $C'$ strictly on its right (resp. left) of height $s\geq t-1$.

\
The proof of the claim is slightly different in the two cases (that is, right or left).

 By \ref {5.4.8}(ix), a right going line carrying a $1$ may be down-going but by only one row.

 Hence the claim for the right.

By \ref {5.4.8}(x), a left going line $\ell$ carrying a $1$ may be down-going by more than one row.  However the latter only occurs if there is a line $\ell'$ carrying a $\ast$, gated at the $t^{th}$ stage, with the same left hand end-point $b$ as $\ell$.  By up-going linkage (\ref {5.4.3}) the right hand end $b'$ of $\ell'$ meets $R_{t-1}$ and since $\ell'$ carries a $\ast$, the box $b'$ lies at the bottom of the column $C'$ in which it is contained, that is in $C'\cap R_{t-1}$.  By adjacency (\ref {5.4.3}) this column $C'$ lies to the left of $C$ which has height $t$.

Hence the claim for the left.



 Let $C'$ be given by the claim.

Adjoin (\textit{only} if necessary) a column $C''$ of height $s$ to be a left (resp. right) neighbour) to $C'$.  Then by \ref {5.4.8}(v) every box in $C\setminus \{b\}$  for this new tableau has a right (resp. left) going line $\ell'$ carrying a $1$, with the important exception that it may carry a $\ast$ if $C'$ lies to the right of $C$ with $\ell'$ being right going and meeting its lowest box $b'$. Then $\ell'$ cannot be $\ell$ which carries a $1$ and indeed $\ell$ must be given by \ref {5.4.7}(i) joining $b,b'$.  In addition $C'$ must have height $t-1$ and have a left neighbour in $\mathscr T_\mathfrak m$.

\end {proof}

\textbf{Remark.}  One can have a left (resp. right) going line from $C$ even if all columns to its left (resp. right) have height $\leq t-1$.

  The first occurs in the array $(1,1,2)$ with $C$ be the column of height $2$.


   The second occurs in the array $(1,2,1)$ with $C$ be the column of height $2$.  The example also illustrates the exception.




\subsection{Line Adjunction to the Last Column}\label{6.6}

To complete our inductive description of $e$ one must show how our construction adjoins lines in the last column to the earlier ones.
This will be carried out in detail in a later paper.  The result is surprisingly simple, though for the moment irritating to establish.  The energetic reader might try it as an exercise.  Here we just concentrate on the results we presently need.

Let $b_{i,j}$ denote the box in $\mathscr D$ lying in $R_i\cap C_j$.


\subsubsection {Lines to the last Column}\label {6.6.1}

Here we just use $\mathscr C$ to denote the set $(C_1,C_2,\ldots,C_k)$ of \textit{all} the columns of $\mathscr D$.  Again we just use $S$ to denote the set of columns excluding the last, that is
$S=(C_1,C_2,\ldots,C_{k-1})$ .

Then $C_k$ is the last column of $\mathscr D$ and following \ref {2.3}, let $\textbf{C}_k$ denote the corresponding column block and $\textbf{S}$
the set of column blocks excluding the last.

Recall \ref {6.1} that $\textbf{M}$ has entries in $\{1,*\}$ which determine  the Weierstrass section $e+V$. These entries all occur in the sub-matrix defining $\mathfrak m$.

 We shall consider the height $s$ of $C_k$ to be a variable (taking non-negative integer values).  Thus it is convenient to write $C_k$ (resp. $\textbf{C}_k$) more precisely as $C_k(s)$ (resp. $\textbf{C}_k(s)$) when this height is $s\in \mathbb N$.  Similarly we write $\textbf{M}$ as $\textbf{M}(s)$.

 By Lemma \ref {5.4.9} the entries of $\textbf{S}$ do not change with $s$.

\subsubsection {The Appearance of $\ast$}\label {6.6.2}

A basic question is to locate the appearance of the $\ast$'s in \textbf{M}. These may be located through the rule given in \ref {5.4.8}.

Here there were two somewhat arbitrary choices. First the rightmost labelling in step $2$ (see third paragraph of \ref {5.3}) secondly the choice made in paragraph four of \ref {5.4.8}.  From this one easily checks that a column block $\textbf{C}_j$ admits no $\ast$ (resp. one $\ast$) exactly when $C_j$ has no (resp. one) left neighbour.  Furthermore in the second case this $\ast$ appear in the last column of $\textbf{C}_j$.  (This last fact is one reason why we made the said choices.)

The rows in which the $\ast$'s appear can also be determined and are given as below.   Let $C,C'$ be neighbouring columns of height $s$, with always the convention that the primed column lies to the right.   Let $b'$ be the box occupying $R_s\cap C'$ and $b$  the adjacent box in the nearest column of height $\geq s$ to the left of $C'$.  Let $i$ (resp. $j$) be the entry of $R_s\cap C$ (resp. $R_s \cap C')$. Set  $\text{r}_{C,C'}=i$ which we call the canonical row associated to the neighbouring pair $C,C'$.

Recall that $C,C'$ are neighbouring columns of height $s$. We say that a pair $C_1,C_1'$ of neighbouring  columns of height $s+1$,  surround the pair $C,C'$, if $C_1'$ is the nearest column strictly to the right of $C'$ of height $\geq s$.

Despite the terminology we \textit{not} require $C_1$ to be to the left of $C$.  For example for the composition $(1,2,1,2)$ the columns of height two surround the columns of height one.  By contrast this false for the composition $(2,1,2,1)$. The required difference in the definition comes from the rightmost labelling adopted in \ref {5.3}.

Let $C,C'$ be a pair of neighbouring columns which do not surround another pair. Then a $\ast$ appears in the column of $\textbf{M}$ corresponding to the last column of $\textbf{C}'$ and on  $\text{r}_{C,C'}^{th}$ row of \textbf{M}.  If $C_1,C_1'$ surround the first pair $C,C'$, then a $\ast$ appears on the last column of $\textbf{C}_1'$ and on  $\text{r}_{C,C'}^{th}$ row of \textbf{M}.  This results from the construction of \ref {5.4.7} and the labelling of \ref {5.4.8}.  This process gives rise to a string of $\ast$'s on the $\text{r}_{C,C'}^{th}$ row of \textbf{M}. These strings are always unbroken and their length is that of the nested sequence of surrounding neighbouring columns.  This string of $\ast$'s may terminate in a $1$ - see examples below.

\textbf{Examples.} The composition $(3,2,1,1,2,3)$ is an example of two-fold surrounding. There is a $\ast$ in row $6$ of \textbf{M} in columns $7,9,10$. If the first column of height three is omitted the last $\ast$ is replaced by a $1$, so ending the string in a $1$, via \ref {5.4.7}(i). For the composition $(1,1,1)$ the singleton string in row one ends in a $1$, through \ref {5.4.7}(ii).

 A further less obvious two-fold surrounding occurs in the composition $(1,3,2,1,2,3)$.  On the other hand, no non-trivial surrounding takes place for the composition $(2,3,1,1,3,2)$.

 For the composition $(1,1,3,2,1,1,2,3)$, there are three strings, one in row $1$ of length one, one on row $6$ of length one which both end in a $1$ and a further string in row $8$ of length three in columns $9,11,14$.

 For the composition $(1,3,1,2,3)$ the columns of height three do not surround the columns of height one, whilst for the composition $(1,2,3,2,1,3)$ the columns of height three do surround the columns of height two and there is a string of length two in row five of \textbf{M}.

\subsection {Changes in $\textbf{C}_k(s)$ with $s \in \mathbb N$} \label {6.7}

We shall show in a subsequent paper that $\textbf{C}_k(s)$ changes in a straightforward manner with $s$.  In particular the first $s-1$ columns of $\textbf{C}_k(s)$ and $\textbf{C}_k(s+1)$ coincide and that the last two columns of the latter follow a rather simple pattern.  For example If $S$ has a column of height $s$, then $\ast$ appears in the $s^{th}$ column of $\textbf{C}_k(s)$ and some unique row, say $\textbf{r}$, then in $\textbf{C}_k(s+1)$ it is replaced by a $1$ in $s^{th}$ column and row $\textbf{r}$.  In other words $\ast$ does not move but is transformed into a $1$.  This is not needed nor proven here.

However the following result is needed and proved here.

\subsection {Appearance of $1$ in the last column of \textbf{M}} \label {6.8}

Here we determine where a $1$ can appear in the last column of \textbf{M}.  In this we take the last column $C_k$ to have height $s:s\geq 0$ and write as before this column  as $C_k(s)$, with the entry in the box $b_k(s)=C_k(s)\cap R_{s}$, being $n$.

For our present purposes we can assume the following:

\

 $S$ has a column $C'$ of height $s$.

 \


  By this assumption, $\ast$ appears in the last column of $\textbf{C}_k(s)$ in some row $\textbf{r}$, so then the line $\ell_{r,n}$ carries a $\ast$.

Then there are two ways that a $1$ can appear in the last column, given by  a line $\ell_{b,b_k(s)}$ carrying a $1$.

Either there is no right extremal box in $R_{s}$, for example in the composition $(1,2,2,2,1,2)$, with $s=2, n=10$.  Then (ii) of \ref {5.4.7} is applied. In terms of $\textbf{M}$ this means that $1$ is placed at the end of the lowest string of $\ast$'s lying above row $\textbf{r}$.  In the example, $r=6$, whilst $\ell_{5,10}$ carries a $1$, which appears in row $5$ and the last column.

Or there is a right extremal box $b$ in $R_{r}:r > s$, for example in the composition  $(1,4,2,1,2)$, with $s=2, n=10$.  Then (i) of \ref {5.4.7} is applied and $\ell_{4,10}$ carries a $1$, which appears in row $4$ and the last column.

\

Let us make more precise the above two cases coming from the joining imposed by \ref {5.4.7} and using the notation used there.  Notice that in the first case we are applying \ref {5.4.7}(ii), in the second case \ref {5.4.7}(i), both after the $s$ stage.

\

In the first case above $b=b_{u-1}$ is joined to $b'_u$, by a line going rightwards and upwards.  Now by definition, $b=b_{u-1}$ lies strictly to the left of  $b'_{u-1}$, and the latter lies in $R_{s}$.

Then by the non-overlapping of ungated lines, $b_u$ lies to the right of $b_{u-1}'$, so $b_u'$ lies strictly to the right of $b_{u-1}'$, itself in a column of height $s$.

\

Let $C'$ the rightmost column of $S$ of height $\geq s$.

\

$(\star)$. In the first case $b=b_{u-1}$ lies strictly to the left of $C'$.  In particular it does not lie in $C'^{\leq s}$.

\

$(\star\star)$. In the second case the extremal box $b$ in $R_{r}:r > s$ must lie in $R_{s+1}$ strictly to the right of $C'$.

\subsection {The Labels $1,\ast$ and the Excluded Roots} \label {6.9}

\subsubsection {The Set of Excluded Roots} \label {6.9.1}

Recall how $\mathfrak u_{\pi'}$ is defined in \ref {2.8}.

Adopt the notation of \ref {2.2} and let $C=C_k$ denote last column of the diagram $\mathscr D$.
Adjoining this last column to the set $S$ of remaining columns the
subspace $\mathfrak u_{\pi'}$ will become smaller and indeed more co-ordinates must be excluded.
These are described explicitly in \ref {2.7}.

\subsubsection {A Key Inclusion} \label {6.9.2}

Let $e$ be obtained by the construction of \ref {5.4}.  It is a sum of co-ordinate vectors.  Let $E$ denote the space spanned by these co-ordinate vectors.

\begin {prop}  $E \subset \mathfrak u_{\pi'}$.

\end {prop}

\begin {proof}  We must show that no co-ordinate vector $x_{i,j}$ in the expression for $e$ is excluded from $B.\mathfrak u_{\pi'}$.  This is achieved by induction on the number of columns of $\mathscr D$.

View $\mathscr D$ as $S$ with the last column $C_k$ adjoined.  Recall (\ref {2.8}) that $\mathfrak u_{\pi'} = \cap_{i=1}^k \mathfrak u_i$.  By Lemma \ref {2.9}(i), it follows that the $\mathfrak u_i:i<k$ contain $\textbf{C}_k$ and so by the induction hypothesis one only needs to show that $E \subset B.{\mathfrak u_k}$.

If $k=1$, there is just one column, so there are no lines between columns.  Hence $e=0$.

The case $k=2$ is a warm-up exercise.  Then there are just two columns and there are excluded co-ordinates only if they have the same height, in which case the excluded co-ordinates lie in the last row.  On the other hand all lines are horizonal and so only $\ast$ appears in the last column.


Let $s$ denote the height of $C_k$, that is of the last column of $\mathscr D$.

Suppose $S$ has no column of height $s$.  Then by Lemma \ref {2.9}(i) no co-ordinates are excluded from $\mathfrak u_k$ and so we can assume that $S$ admits a column of height $s$.

First consider the entries which are excluded in $\textbf{C}_k$.  As noted in \ref {2.7} (the sentence above example 1) these all occur in its last column which is also the last column of $\textbf{M}$.

Recall that $n$ is the largest entry in $\mathscr T_\mathfrak m$.  In the construction of \ref {2.4}, it is shifted to the first column $C_\ell$ of $S$ to the left of $C_k$ of height $\geq s$ and enters in $R_{s+1}$, whilst $C_k^{>s}$ is shifted to the left.
Thus by Lemma \ref {2.5} the excluded co-ordinates of $\mathfrak m$
in the last column of $\textbf{M}$ are the $(i,n)$ with $i$ an entry of $C_\ell^{\leq s}$ or an entry of $C_j:\ell<j<k$.  On the other hand since $S$ has a column of height $s$, the only possible appearances of $1$ in the last column of $\textbf{M}$ obtain from  \ref {6.8}.

Take $C'= C_\ell$ in \ref {6.8}. By $(\star)$ of \ref {6.8}, $i$ is not an entry of $C_\ell^{\leq s}$ and by $(\star\star)$ of \ref {6.8}, it is not an entry of $C_j:\ell<j<k$.

We conclude that the co-ordinates of the last column of $\textbf{M}$ given by the construction of \ref {5.4} lie in $\mathfrak u_{\pi'}$.

Finally we consider entries which are excluded from $\mathfrak u_k$ in the remaining column blocks, that is to say outside $\textbf{C}_k$.  These only arise from the construction of \ref {2.4}, when there is a column $C'$ of height $>s$ between the neighbouring columns $C_\ell,C_k$ of height $s$.  This case is more complicated and the reader needs to take a deep breath!

Let $C$ be the column of height $\geq s$ (possibly $C_\ell$) nearest on the left to $C'$ and different to $C'$. In the construction of \ref {2.4}, $C'^{>s}$ displaces $C^{>s}$ (which may be empty).

Take $i \in C^{\leq s}, j \in C'^{>s}$.  Let $b_0,b_{t'}$ be the blocks in $\mathscr T_\mathfrak m$ for which $i,j$ respectively are their entries.  Let $R_t$ (resp. $R_{t'}$) be the row in which $b_0$ (resp. $b_{t'}$) appears.  Then
$$t\leq s, t'>s. \eqno {(*)}$$

Recall that $e$ is determined by the lines carrying a $1$ in the construction of \ref {5.4}.  These lines are determined by \ref {5.4.7}(i),(ii) with their labels determined by \ref {5.4.8}.

Let us first show that there is no line $\ell:=\ell_{i,j}$ carrying a $1$ obtained by the construction of \ref {5.4.7}(i), recalling that its labelling, being either $1$ or $\ast$, is determined by \ref {5.4.8}.

Set $C^0=C$.  The leftmost column $C^1$ of height $\geq t$ to the right of $C^0$ must have height $t$.  Otherwise there would be a horizontal line $\ell_1$ obtained from step $2$ and carrying a $1$, so either not erased in step $3$, or by \ref {5.4.7}(i), a right and downward going line $\ell'_1$ from $C^0\cap R_t$ still carrying a $1$.  Since there can be only one right going line from $b_0$ carrying a $1$, this forces $\ell_1=\ell$, which contradicts $(*)$.  We conclude that $t$ = ht $C^1$.  Moreover $\ell_1$ must carry a $\ast$, and again \textit{not} be erased for otherwise it would be replaced by a downward going line from $b_0$ carrying $1$, and we would obtain the same contradiction as before.


 This argument proves inductively, that for all $0<i \leq t'-t$, there exist columns $C^i$ of height $t+i-1$ with $C^i$ the leftmost column to the right of $C^{i-1}$ of height $i+t-1$.  Moreover if $b_{i}$ is the  box in $C_i\cap R_{i+t-1}$, then the lines $\ell_i:i=1,2,\ldots,t'-t$ joining $b_0,b_i$ carry a $*$.

 Of course all these columns must lie to the left of $C'$.

  The lines $\ell_i:i=1,2,\ldots,t'-t$ joining $b_0,b_i$ carrying a $*$, permit the existence of the line $\ell$ from $b_0$ to $b_{t'}$ carrying a $1$ to arise from the construction of \ref {5.4.7}(i).  Yet $C^{t'-t}$ whose existence is forced by $\ell$, has height $t'-t+t-1 =t'-1\geq s$, yet is distinct from $C_0=C$.  This contradicts the choice of $C$.

    We conclude that a line $\ell_{i,j}$ constructed by \ref {5.4.7}(i) and carrying a $1$ cannot correspond to a co-ordinate $x_{i,j}$ excluded from $\mathfrak u_k$.

That the line $\ell_{i,j}$ carrying a $1$ cannot come from the construction of \ref {5.4.7}(ii) follows from \ref {6.8}$(*)$.  Indeed recall the notation of \ref {6.8}$(*)$. In this  $x_{i,n}$ obtains from a line carrying a $1$ joining a box $b=b_{u-1}$,  with entry $i$, with the box $b'_{u} = b_{s,k}$ with entry $n$.  On the other hand in the construction of \ref {2.4} the last entry $n$ of $C_k$ displaces $C'^{>s}$ and so $x_{i,n}$ is not excluded from $\mathfrak u_k$, since $b=b_{u-1}$ lies strictly to the left of $C'$.

This concludes the proof of the proposition.

\end {proof}

\textbf{Example.}  Consider the composition $(2,1,1,2,2,1)$.  Here the last column is $C_6$ and $s=1$.  There is a string having two $(*)$'s in $\textbf{r}_3$ of $\textbf{M}$ and two further singleton strings in  $\textbf{r}_6,\textbf{r}_7$.  Then \ref {5.4.7}(ii) gives the line $\ell_{3,8}$ and carries a $1$.    In the above notation $C'$ is just $C_5$ and $b_{u-1}$ lies in $C_2$.  Its entry is $3$ and so the co-ordinate $x_{3,8}$ is not excluded from $\mathfrak u_6$.  Indeed $\mathfrak u_6$, corresponding to $C_3,C_6$ being neighbouring columns of height $1$, excludes only the co-ordinates $x_{4,6},x_{5,8},x_{7,9}$.

\subsubsection {The $\ast$'s are Encircled} \label {6.9.3}

Recall that in \textbf{M} we have used a $1$ (resp. $\ast$) at the co-ordinate place $(i,j)$ exactly when there is a line carrying a $1$ (resp. $\ast$) and going rightwards from a box labelled by $i$ to a box labelled by $j$ in $\mathscr T_\mathfrak m$.  Now use a circle $O$ possibly enclosing a $1$ or $\ast$, at the co-ordinate place $(i,j)$ if $x_{i,j}$ lies in the $\mathfrak h$ stable complement to $\mathfrak u_{\pi'}$ in $\mathfrak m$, that is to say is an excluded co-ordinate in the sense of \ref {2.7}.

\begin {cor}  No circle encloses a $1$.  Every $\ast$ is enclosed by a circle.
\end {cor}

\begin {proof}  The first part follows from the inclusion $e \in \mathfrak u_{\pi'}$ established in \ref {6.9.5}.  The second part follows from the fact established in \ref {5.2} that in every generating invariant $f$ of $S(\mathfrak m^*)^{\mathfrak p'}$, there is just one monomial which survives on restriction to $e+V$ and moreover there is just one term in this monomial which lies in $V$.  Yet by Proposition \ref {2.6}, this monomial vanishes on $\mathfrak u_{\pi'}$.  By the first part the co-ordinate vectors labelled by a $1$ do not vanish on $\mathfrak u_{\pi'}$.  Consequently the remaining co-ordinate vector which carries a $\ast$ must vanish on $\mathfrak u_{\pi'}$, so lies in an excluded co-ordinate. Since every co-ordinate vector lying in $V$ so obtains, the assertion follows.
\end {proof}

\subsubsection {The Upper Bound on $\dim B.\mathfrak u_{\pi'}$ } \label {6.9.4}

In this and the following subsection we shall calculate $\dim B.\mathfrak u_{\pi'}$.  Generally speaking calculating the dimension of a saturation set is notoriously difficult.  Even calculating $B.(\mathfrak n \cap w(\mathfrak n)): w \in W$ can seemingly only be done in type $A$ by appealing to the Robinson-Shensted algorithm. McGovern \cite {Mc} has extended this to classical type.  In type $E_8$, all hell breaks loose.

By \ref {2.8}, it follows that $B.\mathfrak u_{\pi'} \subset B.\mathfrak u_i$ for all $i$, whilst by Proposition \ref {2.6} the closure of the latter is the zero set of the invariant defined by the pair of neighbouring columns defining $\mathfrak u_i$. Thus $B.\mathfrak u_{\pi'} \subset \mathscr N$.  Again since $\mathfrak u_{\pi'}$ is a vector space and $B$ is an irreducible group, the closure of $B.\mathfrak u_{\pi'}$ is irreducible.

Recall that $g$ is defined to be the number of generators of the polynomial algebra $S(\mathfrak m^*)^{\mathfrak p'}$.  The following can be expected but is \textit {not} immediate.

\begin {prop}  $\dim B.\mathfrak u_{\pi'} \leq \dim \mathfrak m -g$.
\end {prop}

\begin {proof} We shall prove the assertion by induction on $g$. Choose an ordering of the $\mathfrak u_i:i=1,2,\ldots,k$, so that $\mathfrak u_k$ is defined by the last pair of neighbouring columns, say $C,C'$. Let $f_k$ denote the corresponding invariant generator.

Set $\mathfrak u_{\pi'}^-= \cap_{i=1}^{k-1} \mathfrak u_i$.  By the induction hypothesis we can assume that $\dim B.\mathfrak u_{\pi'}^- \leq \dim \mathfrak m -(g-1)$.

Since $\dim B.\mathfrak u_{\pi'}^-$ is an irreducible variety, it suffices to show that $f_k$ does not vanish on $B.\mathfrak u_{\pi'}^-$, to obtain the required inequality.

Recall that $f_k$ is a sum of monomials obtained by the standard development of the corresponding Benlolo-Sanderson minor. One of these monomials, say $M(f_k)$, is the product of the root vectors defined by the set of lines forming the unique disjoint union of composite lines joining the boxes in $C$ to the boxes in $C'$.  Any other monomial involves some other root vector, which can be set equal to zero making the contribution to the minor vanish.  Thus $f_k$ vanishes on $\mathfrak u_{\pi'}^-$ only if $M(f_k)$ vanishes on $\mathfrak u_{\pi'}^-$.  Now by the first part of Corollary \ref {6.9.3} every factor in $M(f_k)$, except for the one arising from the line labelled by $\ast$, lies in $\mathfrak u_{\pi'}^-$.

On the other hand, as noted in \ref {6.6.2}, the vector labelled by $\ast$  lies in the last column of $\textbf{C}'$. Yet the only excluded root vectors of the last column of $\textbf{C}'$ come from $\mathfrak u_k$. Thus the vector labelled by $\ast$ must lie in $\mathfrak u_{\pi'}^- $, since $\mathfrak u_k$ is excluded from the intersection defining  $\mathfrak u_{\pi'}^- $. Thus $M(f_k)$  cannot vanish on $\mathfrak u_{\pi'}^-$ and so cannot vanish on $B.\mathfrak u_{\pi'}^-$.  Thus $f_k$ does not vanish on $B.\mathfrak u_{\pi'}^-$.

\end {proof}


\subsubsection {The Lower Bound on $\dim B.\mathfrak u_{\pi'}$ } \label {6.9.5}

The reverse inequality to Proposition \ref {6.9.1} is \textit{not} to be expected.  We start from examples, which illustrate some very fine points in the analysis.

First some notation.

  Let $X$ denote the excluded vectors (cf \ref {2.7}).  We have noted them in \textbf{M} by a $O$ and viewed $O$ as possibly encircling a $1$ or $\ast$. (In \ref {2.4} we had used a $0$, to represent vanishing\footnote {View this change of notation as a mathematical pun.}.) Let $Y$ be the set of co-ordinate vectors in $\mathfrak m$ which carry a $O$ encircling a $\ast$ and set $Z:=X \setminus Y$.  By the construction of \ref {5.4} and Corollary \ref {6.9.3}, the set $Y$ is in natural bijection with the set of generating invariants, so in particular has cardinality $g$.  Let $\textbf{X},\textbf{Z}, \textbf{Y}$ be the subspaces of $\mathfrak m$  generated by $X,Z,Y$ respectively.  Of course $\mathfrak u_{\pi'}$ is just the $\mathfrak h$ stable complement to $\textbf{X}$ in $\mathfrak m$.

  It is the space $\textbf{Z}$ which has to be recovered by the action of $B$ on $\mathfrak u_{\pi'}$.  For $n$ small this is obvious enough as there are plenty of choices to make.  However in general it becomes less obvious particularly because one has to avoid multiple use of elements of $\mathfrak u_{\pi'}$.  What we do is to \textit{cut down} to elements coming from $e$, which might seem foolhardy, yet these elements are strategically placed. In this the $e$ of the Weierstrass section $e+V$ becomes useful.

\

\textbf{Example  1.}  Consider the parabolic defined by composition $(1,3,3,1)$.  From \ref {2.7} and \ref {5.4}, we can determine those entries of $\mathfrak m$ which have a $1, \ast$ or $O$.  For the moment we care only where $\ast,O$ appear.  The former occurs at the co-ordinates labelled by $(4,7),(5,8)$ and are encircled. The remaining $O$ appear at the co-ordinates labelled by $(1,3),(1,4),(2,6),(2,7),(3,7)$.  It is the latter elements we must recover from the action of $B$ on $\mathfrak u_{\pi'}$.


 Apply successively the one parameter subgroups of $B$ generated by the \textit{distinct} elements $x_{i,7}:i=6,5;x_{5,6};x_{2,4};x_{2,3}$ to $\mathfrak u_{\pi'}$, specifically to $x_{3,6},x_{2,5},x_{2,5},x_{1,2},x_{1,2}$ and further successively mod out by the resulting vectors.  In the first set \textit{no} repetitions are allowed, otherwise the one parameter subgroups would not be distinct.  In the second set repetitions \textit{are} allowed.  Notice also that this latter set of vectors \textit{all} come from $e$.  This is not essential but it is extremely convenient because of the ``control'' it gives on the analysis.  This is emphasized by Lemma \ref {6.9.6} and the details in the proof of Theorem \ref {6.9.7}.

The resulting commutators are successively $x_{3,7},x_{2,7},x_{2,6},x_{1,4},x_{1,3}$, which span \textbf{Z}, as required.  In this the action of the fifth one parameter group generator $x_{2,3}$ on $x_{3,6} \in \mathfrak u _{\pi'}$, potentially gives a term which we do not want, namely $x_{2,6}$ but this has already been set to zero by the above moding out through the term obtained from third generator.

We conclude that $Z \subset B.\mathfrak u_{\pi}$ and so the latter has dimension at least $\dim \mathfrak u_{\pi'} + \dim \textbf{Z}= \dim \mathfrak m -g$, as required.


We stress that the duplicity of actions which gave rise to VS pairs, which can rule out the regularity of $e$, does not cause any difficulty here.

 We may also remark that there is an additional co-ordinate vector in $e$, namely $x_{6,8}$, which is not needed here.  This is because $\mathfrak u_{\pi'}$ is $\mathfrak l^-$ stable.

\

\textbf{Example  2.} A further example is provided by the composition \newline $(4,3,2,1,1,2,3,2,1,1,2,3,4)$.  This goes through in a similar fashion.  We leave the details to the reader.  The interest of this case is that in the construction of \ref {5.4} the element $e$ is not regular.   Moreover when we attempt to add VS ones to make $e$ regular, the graph defined by this new element has a cycle (actually a double cycle).   Thus it not assured that we can find an $e$ which is regular. In fact in general one cannot (\ref {6.10.7}).

\subsubsection {Right Going Lines} \label {6.9.6}

 We now use $e$ as constructed \ref {5.4} to show that equality holds in \ref {6.9.4}.  A key point is to have enough right going lines from a given box.

 View $\mathfrak m$ as a subspace of $\textbf{M}$.  Recall that we insert $1$ (resp. $\ast$) into the $(i,j)^{th}$ entry of $\textbf{M}$ exactly when $x_{i,j}$ (resp. $\mathbb C x_{i,j}$) is a summand in $e$ (resp. $V$). This occurs exactly when the line $\ell_{i,j}$ is given by the construction of \ref {5.4} and carries a $1$ (resp. a $\ast$).

 Let $\mathscr P$ denote the set of all pairs of neighbouring columns.  Given $(C,C') \in \mathscr P$, let $X_{C,C'},Y_{C,C'},Z_{C,C'}$ be the subsets of $X,Y,Z$ lying in the column blocks between $\textbf{C},\textbf{C}'$, coming from the pair $C,C'$.  Obviously
$$X= \cup_{(C,C')\in \mathscr P} X_{C,C'}, \quad Y= \cup_{(C,C')\in \mathscr P} Y_{C,C'}.$$

 Thus
 $$Z \subset \cup_{(C,C')\in \mathscr P} Z_{C,C'}. \eqno {(*)}.$$

 Actually in $(*)$ equality holds via induction on rows, but we do not need it.

 \begin {lemma}

  For every element of $x_{i,l} \in Z$, there exists a co-ordinate vector $x_{i,k}$ in $e$ with $k<l$.

   \end {lemma}

 \begin {proof}  By $(*)$ it is enough to prove the assertion for every $Z_{C,C'}$.


 By Lemma \ref {6.5} the labelling of $\textbf{M}$ is unchanged if we remove the columns strictly to the right of $C'$.  Also $Z_{C,C'}$ is unchanged because the entries of the columns strictly to the right of $C'$ are not moved (\ref {2.4}).

 Thus we can assume that $C'$ is the last column  of $\mathscr D_\mathfrak m$. Let $s$ denote its height.

 There is a unique line obtained from the construction of \ref {5.4} joining a box $b\in \mathscr T_\mathfrak m$ to $b_{s,k}$ carrying a $\ast$.  Let $t$ denote entry of $b$ and recall that $n$ is the entry of $b_{s,k}$. Then $\ast$ occurs just once in $\textbf{C}_k$, namely at $x_{t,n}$. Thus
 $$Z_{C,C'}=X_{C,C'}\setminus \{x_{t,n}\}. \eqno{(**)}$$

  It is immediate from the construction of \ref {2.4} that $X_{C,C'}$ does not depend on the columns outside the pair $C,C'$.  Yet between $C,C'$ the lines and their labelling in \textbf{M} can be affected by the columns strictly to the left of $C$ - see \ref {5.4.10}, Example.



The description of  $X_{C,C'}$ is given in \ref {2.7}, and we retain its notation for the columns. One has $C=C_0,C'=C_{u+1}$. \textit{Moreover} the columns of height $<s$ are left unlabelled.


Let $s$ be the common height of $C,C'$.  By our chosen labelling, $C_j;j \in [1,u+1]$ has height $c_j\geq s$ with $c_j>s$ exactly when $j \in [2,u]$. The elements of $X_{C,C'}$ lie in the corresponding column blocks $\textbf{C}_j:j\in [1,u+1]$.

Let $v$ (resp. $v'+1$) designate the top row of the Levi block $\textbf{B}_{j-1}$ (resp. $\textbf{B}_j$), labelled as if they are rows of $\textbf{M}$.  One may note that $v'+1\geq v+c_{j-1}$, with equality if and only if there are no columns of height $<s$, between $C_{j-1}$ and $C_j$.

For all $j\in [2,u]$ (resp. $j=u+1$) that part of $X_{C,C'}$ lying in $\textbf{C}_j$  is a rectangle $\mathscr R_j$  formed from the last $c_j-s$ columns of $\textbf{C}_j$ (resp. last column of $\textbf{C}_{u+1}$), labelled as if columns of $\textbf{M}$ and \textit{with a gap} of size $c_{j-1}-s$ lying between rows $v,v'$ of $\textbf{M}$, starting at row $v+s$ and ending at row $v+c_{j-1}-1$ of \textbf{M}.  (In particular if $j=2$, there is no gap, since $c_1=s$.)  The entries below the gap only occur when $v'>v+c_{j-1}-1$ and correspond to columns of height $<s$ between $C_{j-1},C_j$.

The elements of $X_{C,C'}$ lie in the $\mathscr R_j:j\in [2,u+1]$.

Take an element of this set lying in row $i$ of \textbf{M} and column $l$ of \textbf{M}.

To prove the lemma we have to show that in the construction of \ref {5.4} there is a box $b$ with label $i$ having a right going line with label $1$  meeting a box with label $k<l$.

  To verify this assertion we need just is a small (but subtle) computation remembering how $\mathscr T_\mathfrak m$ is numbered (\ref {2.3}) and to apply Lemmas  \ref {5.4.8}(v), \ref {6.5.3} and \ref {6.9.3}.

For each $j \in [2,u+1]$ let $t_{j-1}$ be the number of rows of \textbf{M} strictly above $\textbf{B}_j$. It is the sum of \textit{all} column heights strictly to the left of $C_j$.

One checks that if $z \in X_{C,C'}$ lies in row $i$ of \textbf{M}, then  $i$ lies in row $R_{i-t_{j-1}}$ of $\mathscr D$.  Moreover if $z$ is above the gap in $\mathscr R_j$, then $i-t_{j-1} \in [1,s]$.

Thus the required  box $b$ with label $i$ lies in $R^s$.

Almost trivially the row of the column block $\textbf{C}_j$ labelled by $i$ as a column of \textbf{M} must be strictly less than an entry of $C_j$.
Thus $b$ must lie in a column strictly to the left of $C_{j}$.

Now by Lemma \ref {5.4.8}(v) there is a right going line $\ell_{b,b'}$ from $b$ to the box $b'$ carrying a $1$ with the one exception  - in which case this line carries a $\ast$.  Of course this is the case described in $(**)$.

Let $k$ be the entry of $b'$. By Lemma \ref {6.5.3}, the line $\ell_{b,b'}$ cannot cross a column of height $>s$, nor can it cross $C_{u+1}$ since this is the last column and so $b'$ must lie in some column strictly to the left of $C_j$, in which case $k<l$, or in $C_j$.

Finally suppose $b' \in C_j$. Recall that by assumption $x_{i,l}$ is an encircled co-ordinate. By the description of $\mathscr R_j$, all the co-ordinates to its right on row $i$ are encircled co-ordinates. Then by Corollary \ref {6.9.3}, it follows that $k$ lies in a column of \textbf{M} lies strictly to the left of that containing $l$,  forcing $k<l$, as required.

Again if $z \in X_{C,C'}$ lies below the gap in $\mathscr R_j$ then $z$ lies in a column strictly between $C_{j-1},C_j$.  These all have height $<s$, that is lie in $R^{s-1}$.

Thus the required  box $b$ with label $i$ lies in $R^{s-1}$ and we conclude exactly as in the previous case.

 \end {proof}

 \subsubsection {Equality of Dimension} \label {6.9.7}

\begin {thm}  $\dim B.\mathfrak u_{\pi'} = \dim \mathfrak m -g$.
\end {thm}

\begin {proof} It suffices by Proposition \ref {6.9.4} to deduce the inequality $\geq$. Retain the notation of \ref {6.9.8}.

Define $e$ through \ref {5.4}.  Recall that by Proposition \ref {6.9.2} and Proposition \ref {2.6}, one has $e \in \mathfrak u_{\pi'}\subset \mathscr N$.  Thus if $e$ is regular in $\mathscr N$, that is if $\dim P.e = \dim \mathfrak m -g$, then the required lower bound of the theorem is immediate.  Yet regularity of $e$ can fail \ref {6.10.9}, Example 1.

As in Example 1 of \ref {6.9.8}, the required lower bound will results from the following

\

\textbf{Claim}.  There exist $\dim \textbf{Z}$ one parameter subgroups of $B$ such that their action on $e$ form tangent spaces which sum to \textbf{Z}.

\

Consider the first column of \textbf{M} containing elements of $Z$ starting from the right. Label this column by $l$. Assume that there is an element of $Z$ in its $i$ row of \textbf{M}, so at the co-ordinate $x_{i,l}$.  By Lemma \ref {6.9.6}, there is a $1$ coming from a co-ordinate vector in  $e$ strictly to the left of a $x_{i,l}$.  This co-ordinate vector is thus $x_{i,k}$ with $k<l$.  Now repeat for all rows of this columns and then for all columns moving leftwards. Then take the $x_{k,l}$ to be the generators of our desired set of one parameter subgroups of $B$.

These generators are pairwise distinct (in a batch corresponding to some fixed $l$) because  the rows (indexed here by $k$) in which $1$'s occur in $\textbf{M}$ are pairwise distinct by Lemma \ref {5.4.6}(vi). Varying the index $l$ labelling the columns in which elements of $Z$ occur, we conclude that they are all distinct, because the columns of \textbf{M} in which the $1$'s occur are  pairwise distinct, again by Lemma \ref {5.4.6}(vi).

This construction sets up a bijection between the elements of $Z$ and the required generators of the one parameter subgroups of $B$. Moreover their action on $e$ give tangent spaces forming a basis of $ \dim \textbf{Z}$.  This proves the claim.

 In Example $1$ of \ref {6.9.8}, the batches corresponding to columns labelled by increasing $l$ are: $x_{2,3};x_{2,4};x_{5,6};x_{5,7},x_{6,7}$. Example $2$ gives much bigger batches which the reader may wish to compute.

Again, as in Example $1$ of \ref {6.9.8}, the $x_{i,k}$ coming from $e$ may be used repeatedly.

Finally let us consider the problem associated with VS pairs.  This follows the analysis given in the example.  In the first instance the action of $\mathfrak n$ on the $1$'s (which by Corollary \ref {6.9.6} lie in $\mathfrak u_{\pi'}$) gives $\textbf{Z}$.  In this we start by successively moding out all the elements in $\mathfrak m$ beginning at the rightmost column and moving leftwards.  Then, as in the example, the second term associated with the adjoint action of a co-ordinate $x_{i,k}\in \mathfrak n$ on $\mathfrak u_{\pi'}$, gives some  $x_{i,l}$, whereas the first term is obtained by the action of the same $x_{i,k}$ on some $x_{i',i}:i'<i$ .  Since necessarily $l>k$ this second term $x_{i,l}$ lies in a column to the right of the one containing the first term, namely $x_{i',k}$.  Thus this new term can be assumed to have already been set to zero.

This concludes the proof of the theorem.
\end {proof}

\textbf{Remark.}  Thus concerning this theorem, the problem associated with VS pairs is rather easily dealt with.  We may also prove this theorem by showing that $E_{VS} \subset \mathfrak u_{\pi'}$.  This does hold and is the reason why we adjoin the second right hand co-ordinate to $e$ - see \ref {1.7}.  However to do this we first need sufficient information on bad VS pairs whose acquisition is a \textit{hard}  matter.  Thus it will be postponed to a subsequent paper.

  \subsubsection {A Component of the Nilfibre} \label {6.9.8}

  Let $M$ be a $\mathbb C$ vector space.  Then $M$ admits a natural action of the multiplicative group $\mathbb C^*$ in which each co-ordinate is multiplied by the same non-zero scalar.  Recall that a closed algebraic subvariety $\mathscr V$ of $M$ is said to be conical if it is stable under this action.  If $\mathscr V$ is conical, we define its projectivisation $[\mathscr V]$ to be the set of equivalence classes for the above action of $\mathbb C^*$ in $\mathscr V \setminus \{0\}$.  It is a projective subvariety of $[M]$ of dimension one smaller than the affine dimension of $\mathscr V$.

  Take $M = \mathfrak m$.  Then the  nilfibre $\mathscr N$ (for invariants with respect to the action of $P'$) is conical and so is any (irreducible) component $\mathscr C$ of the nilfibre.    Again the closure  $\overline {B.\mathfrak u_{\pi'}}$ of $B.\mathfrak u_{\pi'}$ is conical.  Finally let $e+V$ be a Weierstrass section in $\mathfrak m$ (for the action of $P'$.  Then the set $\mathscr V^e$ of all non-zero scalar multiples of $e+v:v \in V\setminus \{0\}$ is conical.

  \begin {cor}   $\overline {B.\mathfrak u_{\pi'}}$ is a component of $\mathscr N$.
  \end {cor}

  \begin {proof}  Since  $\overline {B.\mathfrak u_{\pi'}}$ is closed and irreducible, it is either a component of the nilfibre or a proper subvariety of a component, say $\mathscr C$ of $\mathscr N$. In the latter case by Theorem \ref {6.9.7} one has $\dim [\mathscr C] > \dim [\mathfrak m]-g$.  On the other hand $\dim [\mathscr V^e] =g-1$.  Then by the intersection theory of closed projective subvarieties in projective space \cite [Chap. 1, Section 5, Thm. 6] {Sh}, we conclude that $[\mathscr C] \cap [\mathscr V^e]$ is non-empty.  Translated back to affine space this means that $\mathscr C$ contains an element of  the form $e+v:v \in V\setminus \{0\}$.
  On the other hand by definition of a Weierstrass section, there exists a homogeneous $P'$  invariant $f$ of positive degree such that $f(e+v)$ is a non-zero scalar. (In the present particular case $f$ can be taken to be some Benlolo-Sanderson minor.)  This contradicts that $e+v$ belongs to the nilfibre.  Hence the required assertion.
  \end {proof}

  \textbf{Remark}.  Thus to describe the component $\mathscr C$ of $\mathscr N$ containing, we again \textit{need} the properties of the  Weierstrass section $e+V$, with $e \in \mathscr C$.

  \subsubsection {A Criterion for being Lagrangian} \label {6.9.9}

  One can ask if  $\overline {B.\mathfrak u_{\pi'}}$ for some $w \in W$ and hence involutive, so Lagrangian \cite [Lemma 7.5]{J0}.  By definition a sufficient condition is that $\mathfrak u_{\pi'}$ be complemented in $\mathfrak n$ by a subalgebra, that is of the form
  $\overline{B.\mathfrak n \cap w(\mathfrak n)}$, for some $w \in W$.  However this is not necessary.  A more convenient is fact in the present context is the following

\begin {lemma} Let $\mathscr V$ be a closed irreducible subvariety of $\mathfrak n$.  Then $\mathscr V$ is an orbital variety closure if and only if $\dim \mathscr V =\frac{1}{2} \dim G\mathscr V$.
\end {lemma}

\begin {proof}  Identifying $\mathfrak g^*$ with $\mathfrak g$ through the Killing form, let $\mathscr O \subset \mathfrak g$ be a nilpotent orbit.  After Spaltenstein \cite {S} the intersection $\mathscr O \cap \mathfrak n$ is equidimensional of dimension $\frac{1}{2} \dim \mathscr O$. By say \cite [Thm.2.1]{J0} every component of this intersection takes the form
  $\overline{B.\mathfrak n \cap w(\mathfrak n)}$, for some $w \in W$, that is an orbital variety closure.   From this, only if, is immediate.  Conversely $\overline{G\mathscr V}$ is irreducible, and a union of necessarily finitely many nilpotent orbits. Thus it admits a unique dense orbit $\mathscr O$. Trivially $\overline{\mathscr O}\cap \mathfrak n \supset  \overline{\mathscr O} \cap \mathscr V \supset \mathscr V$. Thus by dimensionality, $\mathscr V$ is a component of the first intersection, hence an orbital variety closure.
\end {proof}

\textbf{Example} Let $P=B$ be the Borel in $SL(3)$.  Then $e$ (resp. $V$) given by the construction of \ref {5.4} is $x_{1,3}$ (resp. the sum of simple root subspaces). Then $\mathfrak u_{\pi'}=\mathbb Cx_{1.3}$ is not complimented in $\mathfrak n$ by a subalgebra.  Moreover it coincides with $\overline{B.\mathfrak u_{\pi'}}$ so has dimension $1$ whilst $\overline{G.\mathfrak u_{\pi'}}$ is the closure of the subregular orbit and has dimension $4$.  Thus $\overline{B.\mathfrak u_{\pi'}}$ is not an orbital variety closure.
   \subsection {VS Pairs} \label {6.10}

  \subsubsection {Recollection of Conventions} \label {6.10.1}

The construction of \ref {5.4} gives lines $\ell_{i,j}$ (resp. $\ell_{r,s}$) labelled by a $1$ (resp. $\ast$).  They define the support of $e$ (resp. $V$) and denoted by  $\supp {e}$ (resp. $\supp {V}$), which gives the Weierstrass section $e+V$.

Thus by definition $x_{i,j}$ (resp. $x_{r,s}$) lies in  $\supp {e}$ (resp. $\supp {V}$) exactly when $\ell_{i,j}$ (resp. $\ell_{r,s}$) is labelled by a $1$ (resp. $\ast$). The entries of $\textbf{M}$ are simialarly defined and so we can view the affine variety $e+V$ as being given by taking $1$ to be in the $(i,j)$ entry and the subspace $\mathbb Cx_{r,s}$ in the $(r,s)$ entry.

\subsubsection {VS Pairs, VS Quadruplets} \label {6.10.2}

A pair $x_{i,j},x_{k,\ell} \in \supp {e}$, with $j \neq k$, is called a VS pair if $x_{i,k}, x_{j,\ell} \notin \supp{V}$.  Then $x_{j,k}$ is a root vector of $\mathfrak p$,
called the connecting element of the VS pair.

Note that we already have $i<j,k<\ell$.

A VS pair  $x_{i,j},x_{k,\ell}$ gives a quadruplet $(i,j,k,\ell)$, which we call a VS quadruplet if $(j,k) \in \supp {V}$.  Notice the number of VS  quadruplets is at most the number of $\ast$'s, that is $g$.

\subsubsection {Bad VS Pairs} \label {6.10.3}

Let $x_{i,j},x_{k,\ell}$ be a VS pair (occurring in $e$). It can happen that there is a co-ordinate $x_{r,\ell}:r\neq k$ (resp. $x_{i,s}:s \neq j$) in $e$ so that the action of $\mathfrak p$ gives $x_{j,\ell}$ (resp. $x_{i,k}$) in such a manner that one of these co-ordinates is recovered in a second fashion.  When this does \textit{not} happen, we say that $x_{i,j},x_{k,\ell}$ is a bad VS pair. A key result in classifying bad VS pairs is the following.

\begin {thm} Let $x_{i,j},x_{k,\ell}$ be a bad VS pair.  Then the  connecting element $x_{j,k}$ comes from a line in $\mathscr T_\mathfrak m$ carrying a $*$.  In particular $(i,j,k,\ell)$ is a VS quadruplet.
\end {thm}

\textbf{Remarks.}  This result will be proved in subsequent paper \cite [4.4.5]{FJ2}.   We might remark that the notion of bad VS pair is not too precisely formulated above and needs for that an induction of columns.  A VS quadruplet associated to a bad VS pair is called a bad VS quadruplet.  A VS quadruplet which is not bad, will be called good.  For example see \ref {6.10.7}.  Classifying the bad VS quadruplets seems to be a difficult problem.

\textbf{Example.}  Consider the parabolic defined by the composition $(2,1,2,1,2)$ and see Figures $10,11$.  It has just one VS quadruplet, namely $(3,4,6,7)$.  It might be thought from the above discussion that it is good quadruplet, because one element of the pair $x_{3,6},x_{4,7}$, namely $x_{4,7}$ can be obtained by the action of $x_{4,6}$ on $x_{6,7}$.  However $x_{4,6}$ is also needed to obtain $x_{3,6}$ from $x_{3,4}$. In fact we need to get the three dimensional space spanned by the co-ordinates with labels $(1,5),(3,6),(4,7)$ by the action of the two dimensional space spanned by the co-ordinates with labels $(3,5),(4,6)$ on the two VS pairs $x_{1,3},x_{5,6}$ and $x_{3,4},x_{6,7}$. Then the additional VS element required is $x_{4,7}$.


\subsubsection {VS Quadruplets} \label {6.10.4}

Let $e_{VS}$ be defined by adjoining the co-ordinate vector $x_{j,l}$ to $e$, for every  $(i,j,k,l)$ that is a VS quadruplet.    In a subsequent paper \cite [4.5.2,4.5.3]{FJ2}, we shall prove the following extension of \ref {6.9.2}.  The result holds if we only adjoin such vectors for bad VS quadruplets, for all VS quadruplets, or for anything in-between.  However (ii),(iii) can fail if we omit vectors corresponding to some bad quadruplets (because we need \cite [Lemma 4.4.4]{FJ2} to hold) and in particular  it fails if we take $E_{VS}=E$, except when $e$ is regular.

\begin {prop}

\

$(i)$. $E_{VS} \subset \mathfrak u_{\pi'}$.

\

$(ii)$. $\dim P.E_{VS} = \dim \mathfrak m-g$.

\

$(iii)$. $\overline{P.E_{VS}}=\overline {B.\mathfrak u_{\pi'}}$.

\end {prop}

\subsubsection {Towards an example of the absence of regular orbits in $\mathscr N$} \label {6.10.5}

\

Take $h \in \mathfrak h$ and $M$ an $h$ invariant subspace of $\mathfrak m$ (for example $V$).  For all $c \in \mathbb C$, set $M_c = \{ m \in M|h.m=cm\}$.  In particular $\mathfrak p_0$ is the $h$ zero weight subspace of $\mathfrak p$.
View $e_{VS}$ as an element of $E_{VS}$ in \textit{general position}.


\begin {lemma}

 Suppose $\overline{B.\mathfrak u_{\pi'}}$ admits a regular element $e'$.  Assume there exists $h \in \mathfrak h$ such that $E_{VS}=(E_{VS}){_{-1}}$
  and that $V_{-1}=0$.
  Then $\mathfrak p_0.e_{VS} = \mathfrak m_{-1}$.
\end {lemma}

\begin {proof} It is clear that $\overline {P.e'}=\overline{B.\mathfrak u_{\pi'}}=\overline {P.E_{VS}}$.  Indeed the first equality follows from the \textit{assumed} regularity of $e'$ and the second equality from assertion (iv) following Proposition \ref {6.10.4}.  In the example of \ref {6.10.7}, we can check (i)- (iii) by hand (see \ref {6.10.7}-digression) so as not to rely on Proposition \ref {6.10.4} which is proved in a subsequent paper in general \cite [Prop. 4.4.11]{FJ2}.

Now being the image of a morphism, $P.E_{VS}$ contains an open set in its closure \cite [Chap. 1, Sect. 6, Thm. 6]{Sh}. Since $P.E_{VS}$ is irreducible, the complement in its closure has strictly lower dimension.  Thus $P.e'$ cannot lie entirely in $\overline {P.E_{VS}} \setminus P.E_{VS}$.
  Thus $P.e' \cap P.E_{VS} \neq \phi$.  The former is a $P$ orbit and the latter a union of $P$ orbits generated by elements of $E_{VS}$.  Thus $P.e' =P.e''$ for some  $e'' \in E_{VS}$, so we can assume $e' \in E_{VS}$ without loss of generality, that is to say we can take $e'=e_{VS}$.  Consequently $\dim \mathfrak p.e_{VS}=\dim P.e_{VS}= \overline{B.\mathfrak u_{\pi'}}= \dim \mathfrak m - \dim V$.  Then by Lemma \ref {3.1}(ii) and \cite [Prop. 4.5.4(ii)]{FJ2}, we conclude that
$$\mathfrak p.e_{VS}\oplus V= \mathfrak m.$$

  Taking the $-1$ eigensubspace with respect to $h$, gives the assertion of the lemma.

\end {proof}

\subsubsection {\textbf{An Example with Good Quadruplets}} \label {6.10.6}

\

The existence of a cycle does \textit{not} guarantee the lack of a dense orbit if the VS quadruplets are chosen badly.  Thus consider the composition $(3,1,1,3,1,2)$.  This has two VS quadruplets namely $(1,4,5,6)$ and $(5,6,9,10)$.  Adjoining the lines $\ell_{4,6},\ell_{6,9}$ to $e$ to obtain $e_{VS}$ gives the cycle $(4,6,10,9,7,4)$.  This has an odd number of lines so we cannot even find an $h \in \mathfrak h$, to make $E_{VS}$ to be a $-1$ eigenvalue subspace for the action of $h$.

In this case one may check (by hand) that both quadruplets are good and that already $e$ is regular.

Thus if we want to show that there is no regular element in a given component of the nilfibre by using the fact that $\mathscr G$ has cycles, then we should at least cut down to bad VS quadruplets.  However this does \textbf{not} ensure that   we choose $h \in \mathfrak h$ such that $h$ has eigenvalue $-1$ on all the co-ordinate vectors in $E_{VS}$.  Rather for the moment this must be verified in each particular case.  On the other hand we shall not need to know that we have cut down to bad quadruplets.  What is more important is that we can choose $h\in \mathfrak h$ with the above eigenvalue property.

\subsubsection {An Example}\label {6.10.7}

We give an example for which the irreducible component $\overline {B.\mathfrak u_{\pi'}}$ has no regular $P$ orbit.  This will \textit{not} use Theorems \ref {6.10.3} and \ref {6.10.4} except as guides, since we can check these assertions for this particular case, by computation.

Take $\mathfrak g$ to be $\mathfrak {sl}(11)$  with the parabolic given by the composition $(2,1,1,2,1,1,2,1)$ and see Figures $8,9$.  One verifies that only $(1,3,6,7)$ and $(6,7,10,11)$ can be bad quadruplets.

Thus $e$ must be augmented by lines joining the pairs $(3,7)$ and $(7,11)$ to give $e_{VS}$.

Since in the present paper we have not proved Proposition \ref {6.10.4}, we make the following short digression to verify its needed conclusions  $(1),(2)$ below in the present special case.

\

\textbf{Digression}

One easily checks that in the present special case that $E_{VS} \subset \mathfrak u_{\pi'}$ in accordance with Theorem \ref {6.10.4}(i).

  Now by Corollary \ref {6.9.8}, $\overline {B.\mathfrak u_{\pi}}$ is a component of $\mathscr N$, so $P$ stable.  Thus we obtain
  $$P.E_{VS}\subset \overline {B. \mathfrak u_{\pi'}}. \eqno {(1)}$$

  In particular via \ref {6.9.7} that $\dim P.E_{VS} \leq \dim \mathfrak m-g$.

 Again in our present special case a straightforward computation shows that
  $$\mathfrak p.e +E_{VS} + V =\mathfrak m. \eqno {(2)}.$$

 Then a standard argument following p. 262 of \cite [Lemma 8.1.1]{D} gives  $\dim P.E_{VS} \geq \dim \mathfrak m-g$.  (For more details, see \cite [Prop. 4.4.11]{FJ2}. Hence (ii), (iii) of Proposition \ref {6.10.4}, in the present special case.

  \

   On the other hand the added pairs pairs $(3,7)$ and $(7,11)$ lead to the cycle $(3,7,11,10,3)$.  This has an even number of lines and \textit{taking care of signs}, we can even choose $h\in \mathfrak h$ to have eigenvalue $-1$ on each root vector occurring in $e_{VS}$.  Let see how general a choice of $h\in \mathfrak h$ can be made.

The lines carrying a $1$, joining the elements of $e$ come in three sets namely $1-3-10-11, \ 2-4-5-8-9, \ 6-7$.  To each such line, for example $\ell_{1,3}$, we assign $-1$ signifying that the corresponding root vector is given the eigenvalue $-1$. These sets define a root system of type $A_3\times A_4 \times A_1$.

 Adjoining $(1,2),(7,11)$ with appropriate signs, we obtain a system of type $A_{10}$.  Now already $(7,11)$ obtains from a VS pair so must assign to it the value $-1$.  \textit{ Rather fortuitously} this also assigns $-1$ to the line $(3,7)$. To the line $(1,2)$ we assign a rational number in general position.

 With this assignment the only roots which have $h$ eigenvalues which are integer are those in the second set above and those coming from the chain
 $1-3-10-11-7-6$.  Here the first $3$ lines carry a $-1$ and the last two a $1$ (because the order of the $\varepsilon$ indices have been reversed).

 From this assignment we may easily calculate, the roots which carry $0$.  They are $(3,6),(7,10)$.  Not surprisingly they are the connecting elements of the two bad VS quadruplets.  All lie outside the Levi factor.  Thus $\mathfrak h$ is complemented by a space $\mathfrak q_0$ in $\mathfrak p_0$ of dimension $2$.

 Similarly the roots coming from the chain $1-3-10-11-7-6$, taking the value $-1$  are $(1,6),(6,10);(6,7),(3,7),(1,3),(10,11),(3,10),(7,11)$ for which the corresponding root vectors lie in $\mathfrak m$. Of these the first two are new, the last six occur in $\supp e_{VS}$.  The space spanned by the first two is just the image under the action of the space $\mathfrak q_0$, spanned by the connecting elements, on $E_{VS}$.

 One may easily check that $V_{-1}$ is empty.

  We conclude that in order for the conclusion of Lemma \ref {6.10.5} to hold we need that $\mathfrak h.e_{VS}=E_{VS}$ or equivalently that $He_{VS}$ be dense in $E_{VS}$.

  Thus we have shown that the existence of a $P$ dense orbit in the component $\overline {B.\mathfrak u_{\pi'}}$ implies the existence of an $H$ dense orbit in $E_{VS}$.

  Yet (in the present case) the roots in $\supp e_{VS}$ form the cycle $(3,7,11,10,3)$ and so by Lemma \ref {6.4}, $E_{VS}$ cannot admit a dense $H$ orbit. Hence

\begin {lemma} For the parabolic defined by the composition $(2,1,1,2,1,1,2,1)$ there is a component of $\mathscr N$ with no dense $P$ orbit.

\end {lemma}

\textbf{Remark.} Of course Proposition \ref {6.10.4} will also hold if we inadvertently include some good VS pairs into the the definition of $E_{VS}$.  However as the example of \ref {6.10.6} shows this may introduce a cycle with an odd number of lines causing the present analysis to fail.

 \subsubsection {A Suggestion for the Set of Components of $\mathscr N$} \label {6.10.8}

 Notice that $\mathbb C[\mathfrak m]$ admits a linear invariant $p_l$  exactly when the diagram describing the Levi factor of $\mathfrak p$ admits two consecutive columns of height $1$.  Suppose that on both sides of these two columns  are columns of height $2$, but not necessarily immediately to the given side. This gives a further invariant generator, say $p$. Then putting $p_l =0$, creates in effect a third column of height $2$, between the given ones.  It follows that $p$ which was originally irreducible \cite [5.3]{FJ}, now factors by Lemma \ref {1.10}.  We suggest that this is the only factorisation of the invariants which can arise and that the resulting components are $B$ saturation sets.

 \textbf{Examples.}  Consider the composition $(2,1,1,2)$. As noted in \cite [5.4.1]{FJ}, there is a linear invariant generator $x_{3,4}$ modulo which the degree $4$ invariant generator factors as a product of $2\times 2$ minors.  The zero set of the bottom minor is exactly the component of $\mathscr N$ described in Corollary  \ref {6.9.8}  as a $B$ saturation set.  By an obvious symmetry one can deduce that the zero set of top minor is a similar $B$ saturation set.

 Consider the composition $(2,1,1,1,2)$.  In this case $\mathscr N$ has three components produced by a similar reasoning. One of these is that described by  Corollary  \ref {6.9.8}, the second by symmetry as above. The third ``middle'' one. is not so obviously a $B$ saturation set, but it is!

 Of course it is completely unobvious that this describes all the irreducible components of $\mathscr N$ in general. Here one may remark that $\mathfrak n \cap \mathscr O$ is equidimensional with components as in \ref {1.7} is a consequence of Bruhat decomposition applied to the Steinberg triple variety.  In this case behind equidimensionality was a disjoint union decomposition of open subsets, which for orbital varieties in type $A$ can be seen rather explicitly in Spaltenstein's original paper \cite {S0}.

 \subsubsection {Computation of Nilpotency Class providing Evidence of Non-regularity} \label {6.10.9}

 Recall $e \in \mathfrak m$ by \ref {5.4}. It belongs to the nilfibre $\mathscr N$, by Lemma \ref {6.1}.

 %

   A result of Spaltenstein, for type $A$,  \cite [Last Cor.]{S0}, asserts that if $\mathscr O$ is a nilpotent orbit, then its intersection with the nilradical  $\mathfrak m$ of a parabolic, is equidimensional.  (In \cite {S} Spaltenstein showed  that the intersection of a nilpotent orbit $\mathscr O$ with the nilradical $\mathfrak n$ of the Borel is equidimensional of dimension $\frac{1}{2} \dim \mathscr O$, in all types.)

   Now $G.e \supset P.e$ and so $G.e \cap \mathfrak m \supset P.e$.  Thus $P.e$, being irreducible, is contained in an irreducible component of $G.e \cap \mathfrak m$.  On the other hand $G.e \cap \mathfrak n$ is equidimensional of dimension $\frac {1}{2} \dim G.e$. Its irreducible components are orbital variety closures.

   We conclude that $\dim P.e \leq \frac {1}{2} \dim G.e$. This can lead to a contradiction with the supposed regularity of $e$, which implies that $\dim P.e = \dim \mathfrak m -g$. Notice further that equality in the first equation implies that $\overline{P.e}$ is an orbital variety closure.

   \

   \textbf{Example 1}.  Let $P$ be defined by  the composition $(3,2,1,1,2,3)$.  The nilpotency class of $e$ is $(5,3,3,1)$.  Thus $\dim G.e = |\Delta|- (4.3+3.2+3.2)$, and so $\frac {1}{2} \dim G.e= \dim \mathfrak n - 12$.  On the other hand if $e$ is regular then $\dim P.e = \dim \mathfrak m-3$. Yet $\dim \mathfrak n -\dim \mathfrak m=8$, forcing $\frac {1}{2} \dim G.e = \dim P.e -1$, which is a contradiction.  Thus $e$ is \textit{not} regular.

   In this example $(4,6,9,11)$ is a bad quadruplet (and the only bad quadruplet).  Thus we must replace to $e$ by $e_{VS}:=e+x_{6,11}$. This element is regular because the extra root does not introduce a cycle\footnote{A general result in this direction will be proved in a subsequent paper.  in the present case the assertion can just be checked.}.  The nilpotency class of $G.e_{VS}$ is easily seen to be $(5,4,2,1)$ and we conclude that $\frac{1}{2} \dim G.e_{VS}= \dim \mathfrak n - 11$, which coincides with $\dim P.e_{VS}$.  In this case we conclude by Lemma \ref {6.9.9} that $\overline{P.e_{VS}}$ and hence   $\overline{B.\mathfrak u_{\pi'}}$ is an orbital variety closure.

   In this case the excluded roots in $\mathfrak u_{\pi'}$ are $\varepsilon_i-\varepsilon_{12}: i \in [1,9], \varepsilon_j-\varepsilon_{8}: j \in [4,7],\varepsilon_6-\varepsilon_7$, which form an additively closed subset of $\Delta^+$.  Thus not only is $\mathfrak u_{\pi'}$ a subalgebra of $\mathfrak n$, but so is its the $\mathfrak h$ stable compliment in $\mathfrak n$. Via \cite [Lemma 7.5]{J0}, this confirms that $\overline{ B.\mathfrak u_{\pi}}$ is an orbital variety closure.

   \

   \textbf{Example 2}.  Let $P$ be defined by  the composition $(1,2,2,1)$.  Then $e$ (resp. $V$) is defined by the co-ordinates $(1,2),(2,4),(5,6)$ (resp. $(4,6),(3,5))$.  One checks that $P.e$ is a dense orbit in $\overline{B.\mathfrak u_{\pi'}}$ and so has dimension $\dim \mathfrak m -2$.  On the other hand the nilpotency class of $e$ is $(3,2,1)$ and so $\frac{1}{2} \dim G.e = \dim \mathfrak n - 4= \dim \overline{B.\mathfrak u_{\pi'}}$.  In this case we conclude by Lemma \ref {6.9.9} that   $\overline{B.\mathfrak u_{\pi'}}$ is an orbital variety closure.  On the other hand the excluded co-ordinates in $\mathfrak u_{\pi'}$ are $(1,3),(2,5),(4,6),(3,5)$ and \textit{do not} form a subalgebra, so we would not have expected $\overline{B.\mathfrak u_{\pi'}}$ to be an orbital variety closure, but contrariwise it is!

   \

   \textbf{Example 3.} Let $P$ be defined by  the composition $(2,1,1,1,2)$.  Then $e$ has nilpotency class $(4,2,1)$, so $\frac{1}{2} \dim G.e= \dim \mathfrak n - 4 =\dim \mathfrak m-2$.  Yet $P.e$ is dense in $\overline{B.\mathfrak u_{\pi'}}$, which is itself of dimension $\dim \mathfrak m -3$.  Then by Lemma \ref {6.9.9}, $\overline{B.\mathfrak u_{\pi'}}$ is not an orbital variety closure,

   \

   \textbf{Example 4.}  Let $P$ be defined by  the composition $(2,1,1,2,1,1,2,1)$ as in \ref {6.10.7} (See Figures $8,9$.)  Then $e_{VS}$ has nilpotency class $(5,4,2)$, so $\frac{1}{2} \dim G.e_{VS}= \dim \mathfrak n - 8 =\dim \mathfrak m-5$.  Yet $P.e_{VS}$ has codimension $1$ in $\overline{B.\mathfrak u_{\pi'}}$, which is itself of dimension $\dim \mathfrak m -6$.  Then by Lemma \ref {6.9.9}, $\overline{B.\mathfrak u_{\pi'}}$ is not an orbital variety closure, nor is $\overline{P.e}$.

   \

\textbf{Example 5.} Consider the example of \ref {6.10.3}.  The nilpotency class of $e$ can be read off from Figure $10$ and is $(3,3,2)$.  Thus the right hand side of \ref {2.3}$(*)$ is $7$, so $\frac{1}{2} \dim G.e = \dim \mathfrak n-7$.  On the other hand if $e$ is regular, one obtains $\dim P.e = \dim \mathfrak m - g= \dim \mathfrak n -6 > \frac{1}{2} \dim G.e$, which is a contradiction.  Thus
$e$ is not regular, as we already suspected.  On the other hand the nilpotency class of $e_{VS}$ is $(4,2,2)$  and so $\frac{1}{2} \dim G.e_{VS} = \dim \mathfrak n-6=\dim P.e_{VS}$, given, as we can check, that $e_{VS}$ is regular. Thus $P.e_{VS}$ is dense in $\overline{B.\mathfrak u_{\pi'}}$, which is furthermore an orbital variety closure by this dimensionality estimate and Lemma \ref {6.9.9}.

   \section {Index of Notation }\label {7}

 Symbols used frequently are given below in the order in which they appear.

 \

 \ref {1}. \quad \ $\mathbb C, [1,n]$.

 \ref {1.1}. \ $G,H,B,\mathfrak g, \mathfrak h, \mathfrak b, \Delta, \Delta ^+, \pi, s _\alpha, W, P_{\pi'},L_{\pi'},M_{\pi'},P'$.

 \ref {1.2}. \ $\mathscr D_\mathfrak m$.

 \ref {1.3}.  \ $e+V$.

 \ref {1.5}.  \ $\mathscr N, (e,h)$.

 \ref {1.6}.  \ $\mathscr T_\mathfrak m, x_{i,j}. L(i,j)$.

 \ref {1.7}.  \ $B.\mathfrak u$.

 \ref {1.8}.  \  $e_{VS}, E_{VS}$.

 \ref {2.1}. \ $\textbf{M}_n, \alpha_{i,j}$.

 \ref {2.2}. \  $C_i,c_i,R_j,R^i$.

 \ref {2.3}. \  $W_{\pi'},w_{\pi'},\textbf{B}_i, \textbf{C}_i$.

 \ref {2.4}. \  $c_i^{\leq s},c_i^{>s},\mathscr F_\mathfrak m (n), \mathscr D_\mathfrak m (n),t_m(n)$.

 \ref {2.5}. \ $w_c(\mathscr T),S(w)$.

 \ref {2.6}. \  $\mathscr T^t, M_s, M_s(\mathfrak m), d(M_s(\mathfrak m))$.

 \ref {2.7}. \  $c_j^i,L^-$.

 \ref {2.8}. \  $\mathfrak u_{\pi'}$.

 \ref {3}. \quad \ $\mathscr N_{reg}$.

 \ref {5.4}. \ $b<b'$.

 \ref {5.4.4}. $b_i,b_i'$.

 \ref {5.4.6}. $b_j''$.

 \ref {6}.  \quad \ \ $\textbf{r}_u,\textbf{c}_v$.

 \ref {6.5.1}. $\hat{\mathscr C}, L(\hat{\mathscr C)}$.

 \ref {6.5.2}. $\hat{\mathscr S}$.

 \ref {6.6}. \  $b_{i,j}$.

 \ref {6.6.1}. $S, \textbf{S},C_k(s),\textbf{C}_k(s), \textbf{M}(s)$.

 \ref {6.6.2}. $\textbf{r}_{C,C'}$.

 \ref {6.9.2}. $X,Y,Z,\textbf{X},\textbf{Y}, \textbf{Z}$.

 \ref {6.9.3}.  $Z_{C,C'}, \mathscr R_j, \textbf{R}_j$.

 \ref {6.9.5}.  $[M]$.

 \ref {6.10.5}. $\mathfrak p_{\neq}$.

 \

\textbf{Acknowledgements.}

This work was the subject of a joint talk on Zoom at Bangaluru, India, 10-12 December 2020. We would like to thank Venkatesh Rajendran for the invitation to speak.  Our talk may be viewed on

\

https://www.youtube.com/watch?v=ALTQj0w2ADM

\

The first author was partly supported at the University of Haifa, arranged by Vladimir Hinich and Anna Melnikov and later at the Weizmann Institute, arranged by Ronen Basri and Gal Binyamini.

\

The authors would like to thank the referee for attentively reading the manuscript, which was no doubt a difficult task, and many pertinent suggestions.

\section*{Illustrations}

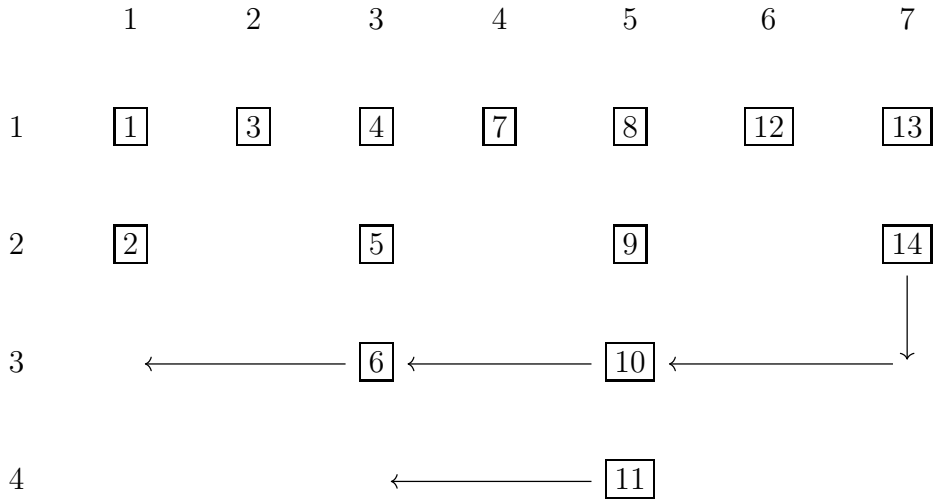
\begin{figure}[H]

\begin{center}
\begin{tikzcd}[row sep=1.8 em,
column sep = 2 em]
  &1&2&3&4&5&6&7\\
1&\fbox{1}&\fbox{3}&\fbox{4}&\fbox{7}&\fbox{8}&\fbox{12}&\fbox{13}\\
2&\fbox{2}&&\fbox{5}&&\fbox{9}&&\fbox{14}
\arrow[->,d, ]
\\
3&{ }&&\fbox{6}\arrow[->,ll,
]&&\fbox{10} \arrow[->,ll, 
]&& { } \arrow[->,ll, ] \\
4&&& { }&&\fbox{11}\arrow[->,ll, 
 ]&&  \\

\end{tikzcd}\\

\caption{The shifting of boxes according to the recipe of Section 2.4.} \label{fig1}

\end{center}
\end{figure}
\begin{figure}[H]
\begin{center}
\begin{tikzcd}[row sep=1.8 em,
column sep = 2 em]
  &1&2&3&4&5&6&7\\
1&\fbox{1}&\fbox{3}&\fbox{4}&\fbox{7}&\fbox{8}&\fbox{12}&\fbox{13}\\
2&\fbox{2}&&\fbox{5}&&\fbox{9}&\\
3&\fbox{6}&&\fbox{10}&&\fbox{14}&\\
4&&&\fbox{11}&&&\\

\end{tikzcd}\\
{}
\caption{The new columns and the resulting Weyl group element in word form: $w=(6,2,1,3,11,10,5,4,7,14,9,8,12)$} \label{fig2}
\end{center}
\end{figure}

  Minimal versus maximal gating for the parabolic define by the composition $(3,1,2,2,1,3)$.
\begin{figure}[H]
\begin{center}
\begin{tikzcd}[row sep=1.8 em,
column sep = 2 em]
\fbox{1}\arrow[-,r,"1"]&\fbox{4}\arrow[-,r,"1"]&\fbox{5}\arrow[-,r,"1"]&\fbox{7}\arrow[-,r,"*"]
&\fbox{9}\arrow[-,r,"1"]&\fbox{10}\\
\fbox{2}\arrow[-,rr,"1"]& &\fbox{6}\arrow[-,r,"*"]&\fbox{8}\arrow[-,rr,"1"]
& &\fbox{11}\\
\fbox{3}\arrow[-,rrrrr,"*"]& &&
& &\fbox{12}\\
\end{tikzcd}\\
\caption{Step $2$}\label{fig8}

\end{center}
\end{figure}
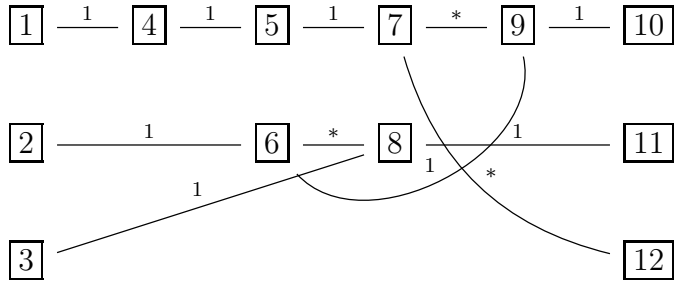
\begin{figure}[H]
\begin{center}
\begin{tikzcd}[row sep=1.8 em,
column sep = 2 em]
\fbox{1}\arrow[-,r,"1"]&\fbox{4}\arrow[-,r,"1"]&\fbox{5}\arrow[-,r,"1"]&\fbox{7}\arrow[-,r,"*"]\arrow[-,rrdd,"*",bend right= 30]
&\fbox{9}\arrow[-,r,"1"]&\fbox{10}\\
\fbox{2}\arrow[-,rr,"1"]& &\fbox{6}\arrow[-,r,"*"] \arrow[-,rru,"1",bend right= 75]&\fbox{8}\arrow[-,rr,"1"]
& &\fbox{11}\\
\fbox{3}\arrow[-,rrru,"1"]& &&
& &\fbox{12}\\
\end{tikzcd}\\
\caption{Step $3$ after minimal gating.}\label{fig9}
\end{center}
\end{figure}
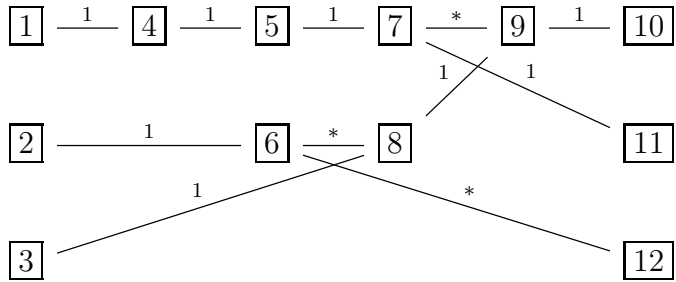
\begin{figure}[H]
\begin{center}
\begin{tikzcd}[row sep=1.8 em,
column sep = 2 em]
\fbox{1}\arrow[-,r,"1"]&\fbox{4}\arrow[-,r,"1"]&\fbox{5}\arrow[-,r,"1"]&\fbox{7}\arrow[-,r,"*"]\arrow[-,rrd,"1"]
&\fbox{9}\arrow[-,r,"1"]&\fbox{10}\\
\fbox{2}\arrow[-,rr,"1"]& &\fbox{6}\arrow[-,r,"*"]\arrow[-,rrrd,"*"]&\fbox{8}\arrow[-,ru,"1"]
& &\fbox{11}\\
\fbox{3}\arrow[-,rrru,"1"]& &&
& &\fbox{12}\\
\end{tikzcd}\\
\caption{Step $3$ after maximal gating.}\label{fig10}
\end{center}
\end{figure}

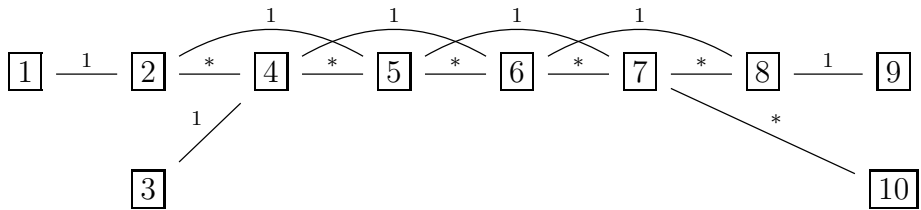
\begin{figure}[H]
\begin{center}
\begin{tikzcd}[row sep=1.8 em,
column sep = 2 em]
\fbox{1}\arrow[-,r,"1"]&\fbox{2}\arrow[-,r,"*"]\arrow[-,rr,"1",bend left= 30]&\fbox{4}\arrow[-,r,"*"]\arrow[-,rr,"1",bend left= 30]
&\fbox{5}\arrow[-,r,"*"]\arrow[-,rr,"1",bend left= 30]&\fbox{6}\arrow[-,r,"*"]\arrow[-,rr,"1",bend left= 30]&\fbox{7}\arrow[-,r,"*"]\arrow[-,rrd,"*"]&\fbox{8}\arrow[-,r,"1"]&\fbox{9}\\
&\fbox{3} \arrow[-,ru,"1"]&&&&&&\fbox{10}\\

\end{tikzcd}\\
\caption{Illustrating the calculation in Example $1$ of 5.4.11}\label{fig12}
\end{center}
\end{figure}

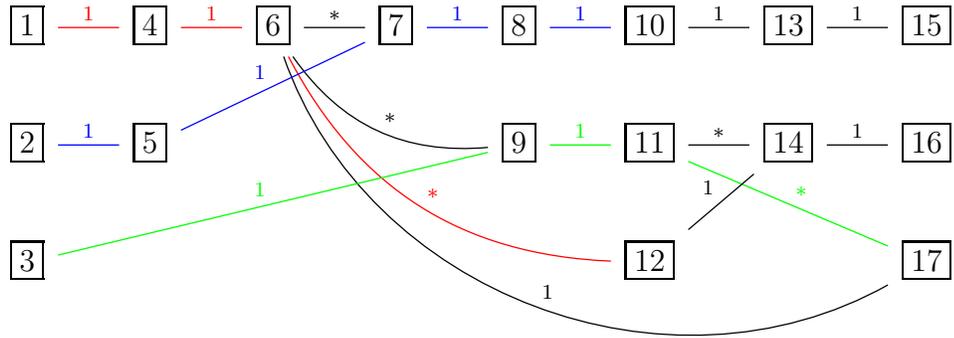
\begin{figure}[H]
\begin{center}
\begin{tikzcd}[row sep=1.8 em,
column sep = 2 em]
\fbox{1}\arrow[-,r,"1",color=red]&\fbox{4}\arrow[-,r,"1",color=red]&\fbox{6}\arrow[-,r,"*"]\arrow[-,rrd,"*",bend right= 30]\arrow[-,rrrdd,"*",bend right= 30, color=red]\arrow[-,rrrrrdd,"1",bend right= 50]&\fbox{7}\arrow[-,r,"1",color=blue]
&\fbox{8}\arrow[-,r,"1",color=blue]&\fbox{10}\arrow[-,r,"1"]&\fbox{13}\arrow[-,r,"1"]&\fbox{15}\\
\fbox{2}\arrow[-,r,"1",color=blue]&\fbox{5}\arrow[-,rru,"1",color=blue]&&&\fbox{9}\arrow[-,r,"1",color=green]&\fbox{11}\arrow[-,r,"*"]\arrow[-,rrd,"*",color=green]&\fbox{14}\arrow[-,r,"1"]&\fbox{16}\\
\fbox{3}\arrow[-,rrrru,"1",color=green]&&&&&\fbox{12}\arrow[-,ru,"1"]&&\fbox{17}\\
\end{tikzcd}\\
\caption{The composite lines of the disjoint union between the first pair of neighboring column of height $3$ are in red, blue and green. }
\end{center}
\end{figure}
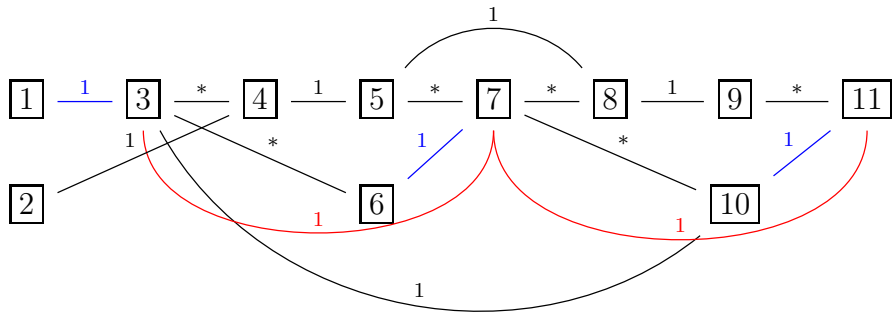
\begin{figure}[H]
\begin{center}
\begin{tikzcd}[row sep=1.4 em,
column sep = 1.8 em]
\fbox{1}\arrow[-,r,"1",color=blue]&\fbox{3}\arrow[-,r,"*"]\arrow[-,rrd,"*"]\arrow[-,rrrrrd,"1",bend right= 50]\arrow[-,rrr,"1",bend right= 90, color=red ]&\fbox{4}\arrow[-,r,"1"]&\fbox{5}\arrow[-,r,"*"]\arrow[-,rr,"1",bend left= 50]&\fbox{7}\arrow[-,r,"*"]\arrow[-,rrd,"*"]\arrow[-,rrr,"1",bend right= 90, color=red ]&\fbox{8}\arrow[-,r,"1"]&\fbox{9}\arrow[-,r,"*"]&\fbox{11}\\
\fbox{2}\arrow[-,rru,"1"]&&&\fbox{6}\arrow[-,ru,"1",color=blue]&&&\fbox{10}\arrow[-,ru,"1",color=blue]& \\

\end{tikzcd}\\

\caption{Illustrating the example of 6.10.7.  The bad  VS-pairs  $(x_{1,3},x_{6,7}), (x_{6,7},x_{10,11})$ are given by the  blue lines  and VS elements
$x_{3,7}$ and $x_{7,11}$ by the red lines.}
\end{center}
\end{figure}

%
\begin{figure}[H]
\begin{center}
 \begin{tikzpicture}
 \matrix [matrix of math nodes,left delimiter=(,right delimiter=)] (m)
{
1 & 0 & \blue{1} & 0 & 0 & 0 & 0 & 0 & 0 & 0 & 0 \\
0 & 1 & 0 & 1 & 0 & 0 & 0 & 0 & 0 & 0 & 0 \\
0 & 0 & 1 & 0 & 0 & 0 & \red{1} & 0 & 0 & 1 & 0 \\
0 & 0 & 0 & 1 & 1 & 0 & 0 & 0 & 0 & 0 & 0 \\
0 & 0 & 0 & 0 & 1 & 0 & * & 1 & 0 & 0 & 0 \\
0 & 0 & 0 & 0 & 0 & 1 & \blue{1} & 0 & 0 & 0 & 0 \\
0 & 0 & 0 & 0 & 0 & 0 & 1 & * & 0 & * & \red{1} \\
0 & 0 & 0 & 0 & 0 & 0 & 0 & 1 & 1 & 0 & 0 \\
0 & 0 & 0 & 0 & 0 & 0 & 0 & 0 & 1 & 0 & * \\
0 & 0 & 0 & 0 & 0 & 0 & 0 & 0 & 0 & 1 & \blue{1} \\
0 & 0 & 0 & 0 & 0 & 0 & 0 & 0 & 0 & 0 & 1 \\
};
\draw[
fill=blue, opacity=0.5] (m-2-1.south west) rectangle (m-1-2.north east);
\draw[
fill=red, opacity=0.5] (m-3-3.south west) rectangle (m-3-3.north east);
\draw[
fill=red, opacity=0.5] (m-4-4.south west) rectangle (m-4-4.north east);
\draw[
fill=blue, opacity=0.5] (m-6-5.south west) rectangle (m-5-6.north east);
\draw[
fill=red, opacity=0.5] (m-7-7.south west) rectangle (m-7-7.north east);
\draw[
fill=red, opacity=0.5] (m-8-8.south west) rectangle (m-8-8.north east);
\draw[
fill=blue, opacity=0.5] (m-10-9.south west) rectangle (m-9-10.north east);
\draw[
fill=red, opacity=0.5] (m-11-11.south west) rectangle (m-11-11.north east);
 \end{tikzpicture}\\

\caption{The example of 6.10.7 in matrix form. Here $e_{VS}$ is represented by the $1$'s with the additional VS elements in red. The $\ast$'s represent $V$.}
\end{center}
\end{figure}

\newpage
Illustrating the example of Section 6.10.3.

An unexpected bad quadruplet in the parabolic defined by the composition $(2,1,2,1,2)$:
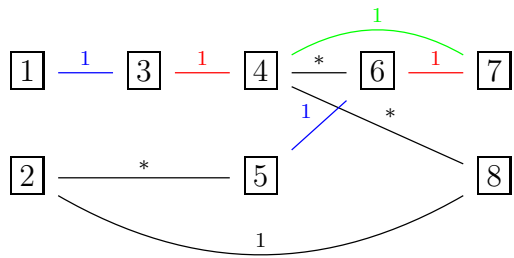
\begin{figure}[H]
\begin{center}
\begin{tikzcd}[row sep=1.4 em,
column sep = 1.8 em]
\fbox{1}\arrow[-,r,"1",color=blue]&\fbox{3}\arrow[-,r,"1",color=red ]&\fbox{4}\arrow[-,r,"*"]\arrow[-,rrd,"*",bend left= 0]\arrow[-,rr,"1",bend left= 30,color=green]&\fbox{6}\arrow[-,r,"1",color=red]&\fbox{7}\\
\fbox{2}\arrow[-,rr,"*"]\arrow[-,rrrr,"1",bend right= 30]&&\fbox{5}\arrow[-,ru,"1",color=blue]&&\fbox{8} \\

\end{tikzcd}\\

\caption{The red lines mark a seemingly good VS quadruplet  $(3,4,6,7)$. Yet $x_{4,6}$ is also needed to obtain $x_{3.6}$ from $x_{3,4}$ in the  VS-pair  $(x_{1,3},x_{5,6})$ given by the  blue lines. Thus one needs to adjoin the VS element $x_{4,7}$ to obtain $e_{VS}$, marked in green. }
\end{center}
\end{figure}
The matrix of $e_{VS}$
\begin{figure}[H]
\begin{center}
 \begin{tikzpicture}
 \matrix [matrix of math nodes,left delimiter=(,right delimiter=)] (m)
{
1 & 0 & \blue{1} & 0 &{\encircled{0}} & 0 & 0 & 0 \\
0 & 1 & 0 & 0 &{\encircled{*}}& 0 & 0 & 1 \\
0 & 0 & 1 & \red{1} & {\encircled{0}} & 0& 0 & 0  \\
0 & 0 & 0 & 1 & 0 & {\encircled{0}} & \green{1} & {\encircled{*}} \\
0 & 0 & 0 & 0 & 1 & \blue{1} & 0& {\encircled{0}}  \\
0 & 0 & 0 & 0 & 0 & 1 & \red{1}& {\encircled{0}} \\
0 & 0 & 0 & 0 & 0 & 0 & 1 &  0\\
0 & 0 & 0 & 0 & 0 & 0 & 0 & 1 \\
};
\draw[
fill=blue, opacity=0.5] (m-2-1.south west) rectangle (m-1-2.north east);
\draw[
fill=red, opacity=0.5] (m-3-3.south west) rectangle (m-3-3.north east);
\draw[
fill=blue, opacity=0.5] (m-5-4.south west) rectangle (m-4-5.north east);
\draw[
fill=red, opacity=0.5] (m-6-6.south west) rectangle (m-6-6.north east);
\draw[
fill=blue, opacity=0.5] (m-8-7.south west) rectangle (m-7-8.north east);

\end{tikzpicture}\\

\caption{The example of 6.10.3 in matrix form. Here $e_{VS}$ is represented by the $1$'s with the additional VS element in green. The $\ast$'s represent $V$. The excluded root vectors (see (2.6)) are encircled. }

\end{center}
\end{figure}

\end{document}